\documentclass[10pt, letterpaper, reqno]{amsart}
\usepackage{indentfirst} 
\usepackage{amssymb}
\usepackage{mathtools} 
\usepackage{mathabx} 
\usepackage[normalem]{ulem} 
\usepackage{amsthm}  
\usepackage{thmtools} 
\usepackage[backref=page,colorlinks=true]{hyperref} 
\usepackage[font=footnotesize]{caption}
\usepackage{enumitem} 
\usepackage[usenames,dvipsnames]{xcolor} 
\usepackage{tikz}   
\usetikzlibrary{decorations.markings,intersections}
\usepackage{ifthen}

\renewcommand*{\backref}[1]{}
\renewcommand*{\backrefalt}[4]{\quad \tiny 
    \ifcase #1 (\textbf{NOT CITED.})%
    \or        (Cited on page~#2.)%
    \else      (Cited on pages~#2.)%
    \fi}

\makeatletter

\def\myMRbibitem{\@ifnextchar[\my@lbibitem\my@bibitem}

\def\mybiblabel#1#2{\@biblabel{{\hyperref{http://www.ams.org/mathscinet-getitem?mr=#1}{}{}{#2}}}}

\def\myhyperanchor#1{\Hy@raisedlink{\hyper@anchorstart{cite.#1}\hyper@anchorend}}

\def\my@lbibitem[#1]#2#3#4\par{%
    \item[\mybiblabel{#2}{#1}\myhyperanchor{#3}\hfill]#4%
    \@ifundefined{ifbackrefparscan}{}{\BR@backref{#3}}%
    \if@filesw{\let\protect\noexpand\immediate
       \write\@auxout{\string\bibcite{#3}{#1}}}\fi\ignorespaces%
}

\def\my@bibitem#1#2#3\par{%
    \refstepcounter\@listctr
    \item[\mybiblabel{#1}{\the\value\@listctr}\myhyperanchor{#2}\hfill]#3%
    \@ifundefined{ifbackrefparscan}{}{\BR@backref{#2}}%
    \if@filesw\immediate\write\@auxout
        {\string\bibcite{#2}{\the\value\@listctr}}\fi\ignorespaces%
}

\makeatother

\declaretheorem{theorem}

\declaretheorem[numberwithin=section, name=Theorem]{otherthm}
\declaretheorem[sibling=otherthm]{lemma}
\declaretheorem[sibling=otherthm]{corollary}
\declaretheorem[sibling=otherthm]{proposition}
\declaretheorem[sibling=otherthm]{scholium}
\declaretheorem[sibling=otherthm,style=remark]{fact}

\declaretheorem[numbered=no, style=remark]{remark}
\declaretheorem[numbered=no, style=remark, name=Acknowledgement]{ack}

\hypersetup{bookmarksdepth = 3} 
\numberwithin{equation}{section}         

\setlist[enumerate,1]{label={\upshape(\alph*)},ref=\alph*}
\setlist[enumerate,2]{label={\upshape(\arabic*)},ref=\arabic*}

\newcommand{\R}{\mathbb{R}}

\newcommand{\Z}{\mathbb{Z}}
\newcommand{\N}{\mathbb{N}}
\newcommand{\E}{\mathbb{E}}
\newcommand{\F}{\mathbb{F}}
\newcommand{\G}{\mathbb{G}}
\newcommand{\cC}{\mathcal{C}}
\newcommand{\cD}{\mathcal{D}}
\newcommand{\cI}{\mathcal{I}}
\newcommand{\cL}{\mathcal{L}}
\newcommand{\cM}{\mathcal{M}}

\newcommand{\cV}{\mathcal{V}}

\renewcommand{\epsilon}{\varepsilon}
\renewcommand{\phi}{\varphi}
\renewcommand{\emptyset}{\varnothing}
\renewcommand{\setminus}{\smallsetminus}
\newcommand{\id}{\mathrm{id}}
\newcommand{\SL}{\mathrm{SL}}
\newcommand{\GL}{\mathrm{GL}}
\newcommand{\SO}{\mathrm{SO}}
\renewcommand{\O}{\mathrm{O}}
\newcommand{\Gr}{\mathrm{Gr}}

\newcommand{\Aut}{\mathrm{Aut}}
\newcommand{\Sec}{\mathrm{Sec}}
\newcommand{\Diff}{\mathrm{Diff}}
\newcommand{\Homeo}{\mathrm{Homeo}}
\DeclareMathOperator{\interior}{int}
\newcommand{\widecup}{\enspace{\cup}\enspace}
\newcommand{\widecap}{\enspace{\cap}\enspace}
\newcommand{\transverse}{\;\;\makebox[0pt]{$\top$}\makebox[0pt]{\small $\cap$}\;\;}
\DeclareMathOperator{\Ker}{Ker}
\renewcommand{\Im}{\mathop{\mathrm{Im}}}
\DeclareMathOperator{\sgn}{sgn}

\newcommand{\bundle}[4]{{#4} \hookrightarrow {#1} \overset{{#3}}{\to} {#2}}

\newcommand{\arxiv}[1]{Preprint \href{http://arxiv.org/abs/#1}{arXiv:{#1}}}
\newcommand{\doi}[1]{\href{http://dx.doi.org/#1}{doi:{#1}}}

\begin{document}

\title{Cocycles of isometries and denseness of domination}
\author{Jairo Bochi}
\address{Facultad de Matem\'aticas, Pontificia Universidad Cat\'olica de Chile}
\urladdr{\href{http://www.mat.uc.cl/~jairo.bochi}{www.mat.uc.cl/$\sim$jairo.bochi}}
\email{\href{mailto:jairo.bochi@gmail.com}{jairo.bochi@gmail.com}}
\date{August 26, 2014}

\begin{thanks}
{The author was partially supported by project Fondecyt 1140202 and by the Center of Dynamical Systems and Related Fields ACT1103.}
\end{thanks}

\begin{abstract}
We consider the problem of approximating a linear cocycle (or, more generally, a vector bundle automorphism) over a fixed base dynamics by another cocycle admitting a dominated splitting. We prove that  the possibility of doing so depends only on the homotopy class of the cocycle, provided that the base dynamics is a minimal diffeomorphism and the fiber dimension is least $3$. This result is obtained by means of a general theorem on the existence of almost invariant sections for fiberwise isometries of bundles with compact fibers and finite fundamental group. The main novelty of the proofs is the use of a quantitative homotopy result due to Calder, Siegel, and Williams.
\end{abstract}

\maketitle


\section{Introduction}

\subsection{A dynamical interplay}

This paper deals with the dynamics of certain classes of fiber bundle automorphisms.
In particular, these include skew-products $g(x,y) = (f(x),g_x(y))$ 
acting on trivial bundles $X \times Y$.
If the map $x \mapsto g_x$ takes values in a specific group $G$
of transformations of the fiber $Y$, it is called a \emph{$G$-cocycle}.

The first class we consider consists of vector bundle automorphisms,
which in particular include linear cocycles.
To investigate them, it is often useful to consider 
induced automorphisms on other (not necessarily linear) fiber bundles.
The simplest example is the projectivization of a vector bundle automorphism.
A related linearly-induced automorphism of a fiber bundle with $\SO(m)$ fibers
was used in an essential way by V.I.~Oseledets in the proof of his celebrated theorem: 
see \cite[p.~229]{Oseledets}.
Quoting \cite{Selgrade}, compactness of the fibers 
``allows the use of techniques not available for the vector bundle''.

Nevertheless, it is also useful to consider linearly-induced automorphisms 
on bundles with non-compact fibers, especially if these have some extra structure.
For example, 
a $\SL(2,\R)$-cocycle induces a cocycle of M\"obius transformations of the complex half-plane,
which are isometries with respect to the hyperbolic metric,
and many linear-algebraic properties of the former cocycle
can be understood in terms of geometric properties of the latter.
A far reaching extension of this interplay
is revealed by the Karlsson--Margulis theorem \cite{KM}
on cocycles of isometries of spaces of nonpositive curvature,
which yields Oseledets theorem as a corollary. 
The remarkable generality and simplicity of the Karlsson--Margulis theorem 
have instant appeal and justify the study of cocycles of isometries 
for its own sake.

\medskip

The notion of \emph{dominated splittings} 
is central to the dynamics of vector bundle automorphisms
and is a major motivation for this paper.
It basically consists on a projective form of hyperbolicity,
and it is equivalent to ordinary uniform hyperbolicity in the case 
of $\SL(2,\R)$-cocycles.
The term ``domination'' was coined by R.~Ma\~{n}\'{e} in the 1970's,
although the concept was actually introduced earlier
in differential equations theory under the name ``exponential separation'': see \cite{Sambarino,Palmer} and references therein. 
Dominated splittings are intrinsically related to chain recurrence properties of the 
induced projectivized automorphism \cite{Selgrade,CK},
and can also be characterized in terms of separation between singular values \cite{BG,Morris}.
Since Ma\~{n}\'{e}, dominated splittings continue to play a important role 
in differentiable dynamics on compact manifolds: see \cite{BDV,Sambarino}.

The issue we are concerned with here is denseness of domination: 
when can a given vector bundle automorphism 
be approximated by another having a dominated splitting?
Under reasonable assumptions, 
we reduce this question to a problem about 
the existence of almost invariant sections 
for fiberwise isometries, which we them solve in a much greater generality.
That general result is the core of this paper. 
It turns out to have other applications: 
we use it to characterize almost coboundaries on compact Lie groups with finite center.

\medskip

Other very general constructions of invariant and almost invariant sections
for cocycles of isometries appear in the works \cite{CNP,BN_geometric};
these rely on nonpositive curvature and are highly geometrical.
By contrast, the isometries considered in this paper act on compact fibers,
whose geometries are less favorable: 
for example, shortest geodesics between pairs of points are not necessarily unique.
Actually the arguments developed here are much more topological than geometrical,
and use as a crucial ingredient beautiful results on \emph{quantitative homotopy}
by Calder, Siegel, and Williams \cite{CS80,SW89}.

Let us proceed with precise statements.

\subsection{Domination and the problem of denseness}

Let $X$ be a compact Hausdorff space.
Let $m \ge 2$ be an integer, and let $\E$ be a $m$-plane bundle over $X$,
that is, a real vector bundle with base space $X$
and fibers of dimension $m$.
We endow $\E$ with a Riemannian norm.

If  $f \colon X \to X$ is a homeomorphism, we let 
$\Aut(\E,f)$ denote the space of automorphisms of $\E$ fibering over $f$,
endowed with the uniform (i.e.\ $C^0$) topology.

When the vector bundle is trivial, 
that is $\E = X \times \R^m$, there is an identification $\Aut(\E,f) = C(X,\GL(m,\R))$;
indeed every automorphism is of the form $(x,v) \mapsto (f(x),A(x)v)$ for some 
continuous map $A \colon X \to \GL(m,\R)$,
which is called a \emph{linear cocycle}.

\medskip

Consider a splitting $\E = \E^1 \oplus \E^2 \oplus \cdots \oplus \E^k$ of the bundle $\E$ as a sum of proper nontrivial subbundles $\E^i$.
This splitting is called \emph{dominated} with respect to an automorphism $A \in \Aut(\E,f)$
if each subbundle $\E^i$ is $A$-invariant
and moreover there is a constant integer $\ell \in \N$ such that
for all $x \in X$, all $i \in \{1,\dots,k-1\}$, and all unit vectors 
$v_i \in \E^i_x$, $v_{i+1} \in \E^{i+1}_x$, we have
$$
\| A^\ell (v_i) \| > \|A^\ell (v_{i+1})\| \, .
$$
That is, up to replacing $A$ by a power, 
any vector in $\E^i_x$ is relatively more expanded than any vector in $\E^{i+1}_x$.
We also say that $\E^i$ \emph{dominates} $\E^{i+1}$.

\medskip

Let us consider the base dynamics $f$ as fixed.
An important feature of domination is \emph{openness}:
the set of automorphisms admitting a dominated splitting is
open in $\Aut(\E,f)$.
On the other hand, domination is not dense in general. 
If $f$ has a periodic point $x$ of period $p$ such that 
the restriction of the power $A^p$ to the fiber $\E_x$
has exactly two eigenvalues of maximum absolute value, and these eigenvalues are non-real, then the automorphism $A$ cannot admit a dominated splitting whose top subbundle $\E^1$ is one-dimensional. Such a condition is open in $\Aut(\E,f)$.
With this kind of reasoning we can exhibit nonempty open subsets of $\Aut(\E,f)$
formed by automorphisms that admit no dominated splitting at all,
provided $f$ has sufficiently many periodic points.

Different obstructions to domination
may be due to topological reasons:
sometimes the homotopy type of $A$ forbids
the existence of an invariant splitting, and in particular, of a dominated one.
(See \S~\ref{ss.two_types} for an example.)

Suppose that the base dynamics $f$ is minimal (and the base space is infinite), and so periodic orbit obstructions do not arise. Our first main result basically states that all robust obstructions to domination are topological, provided the (linear) dimension is at least $3$. 
The precise statement is as follows:

\begin{theorem}\label{t.densedom}
Let $f \colon X \to X$ be a minimal diffeomorphism of 
a compact manifold $X$ of positive dimension.
Let $\E$ be a $m$-plane bundle over $X$, where $m \ge 3$.
Then for each fibered homotopy class $\cC \subset \Aut(\E,f)$,
\begin{enumerate}
\item\label{i.obstruction}
either no automorphism in $\cC$ has a proper nontrivial invariant subbundle;
\item\label{i.densedom}
or there is an open and dense subset $\cD \subset \cC$ such that all automorphisms in $\cD$ have a dominated splitting.
\end{enumerate}
\end{theorem}

Here the \emph{fibered homotopy class} of an automorphism is its path-connected component
in $\Aut(\E,f)$; the corresponding paths are called \emph{fibered homotopies}.
So the theorem states that if an automorphism $A$ is fibered homotopic to another 
having a nontrivial continuous invariant field of planes
then a perturbation of $A$ has a dominated splitting.
In particular, domination is either empty or dense 
inside each fibered homotopy class.

Notice that Theorem~\ref{t.densedom} requires $f$ to be a diffeomorphism.
(In this paper, we assume all manifolds to be $C^\infty$, without boundary, and paracompact,
and all diffeomorphisms to be $C^\infty$.)
Although this assumption should be stronger than necessary,
it is technically very convenient for certain parts of the construction 
(especially those in Appendix~\ref{s.appendix}),
and so we have not tried to optimize it.

\medskip

More information about 
the classes $\cC$ of type~(\ref{i.obstruction}) in Theorem~\ref{t.densedom} is available:
generically in $\cC$ the automorphism is uniformly subexponentially quasiconformal
(by a result of \cite{B_Studia}),
and densely in $\cC$ there is an invariant conformal structure (by a result of \cite{BN_elementary});
see \S~\ref{ss.dom_vs_conf} for details.

Theorem~\ref{t.densedom} does not hold in dimension $m=2$,
because in this case there exists another obstruction to domination
related to the rotation number: see \S~\ref{ss.dim2}.

For examples, we refer the reader to \S\S~\ref{ss.two_types} and \ref{ss.indices},
where we show that $\Aut(\E,f)$ can indeed contain classes of 
both types (\ref{i.obstruction}) and (\ref{i.densedom}),
and that a class of type (\ref{i.densedom}) can contain different types of domination.

\medskip

Other results on denseness of domination may be found in the papers 
\cite{Million} (for autonomous linear differential equations),
\cite{Cong} (for bounded measurable cocycles), 
\cite{ABD1,ABD2} (for continuous $\SL(2,\R)$-cocycles over uniquely ergodic dynamics), 
\cite{FJZ} (for H\"older-continuous $\SL(2,\R)$-cocycles over generic irrational flows on the two-torus),
and \cite{AJS} (for analytic complex-valued cocycles over rotations).

\medskip

We next describe the setting of fiberwise isometries,
which we will later relate to Theorem~\ref{t.densedom}.

\subsection{Fiberwise isometries and almost invariant sections}\label{ss.fib_isom_sections}


Let $X$ be a compact Hausdorff space, and let $Y$ be a manifold. 
A fiber bundle $\bundle{Z}{X}{p}{Y}$ 
is called \emph{fiberwise smooth} if 
its structural group is formed by diffeomorphisms of $Y$.
Then each fiber $Z_x \coloneqq p^{-1}(x)$ has a manifold structure
and is diffeomorphic to~$Y$.
If $g \colon Z \to Z$ is a bundle automorphism then there exists
a homeomorphism $f \colon X \to X$ such that $g$ diffeomorphically maps
the fiber $Z_x$ to the fiber $Z_{f(x)}$. We say that $g$ fibers over $f$.
The set of all automorphisms of $Z$ is denoted by $\Aut(Z)$,
and the set of automorphisms fibering over a given $f \in\Homeo(X)$
is denoted by $\Aut(Z,f)$.


The \emph{vertical tangent bundle} is the union $\bigsqcup_{x\in X} TZ_x$ 
of the tangent bundles of the fibers, endowed with the obvious vector bundle structure.
A \emph{fibered Riemannian structure} on the fiberwise smooth bundle $\bundle{Z}{X}{p}{Y}$ is
a continuous field of positive definite quadratic forms on the vertical tangent bundle
whose restriction to each $TZ_x$ is a (smooth) Riemannian metric on the manifold $Z_x$.
Such structures always exist.
An automorphism $g \in \Aut(Z)$ is called a \emph{fiberwise isometry} if it preserves 
a given fibered Riemannian structure.


Let $\Sec(Z)$ denote the space of all sections of $Z$, that is, 
all continuous maps $\sigma \colon X \to Z$ such that $p \circ \sigma = \id_X$.
The \emph{distance} between $\sigma$, $\sigma' \in \Sec(Z)$ is defined as
\begin{equation}\label{e.def_distance}
\mathrm{d}(\sigma, \sigma') \coloneqq \sup_{x \in X} \mathrm{d}_x( \sigma(x), \sigma'(x)),
\end{equation}
where $\mathrm{d}_x$ denotes Riemannian distance on the fiber $Z_x$.
This makes $\Sec(Z)$ a metric space.


Throughout this paper, we denote the unit interval as
$$
I \coloneqq [0,1] \, .
$$
We say that $\sigma$, $\sigma' \in \Sec(Z)$ are \emph{fibered homotopic}
if they are homotopic through sections, that is,
there exists a continuous curve $t \in I \mapsto \sigma_t \in \Sec(Z)$ 
(called a \emph{fibered homotopy})
from $\sigma_0 = \sigma$ to $\sigma_1 = \sigma'$. 
If the fibered homotopy is of the form $\sigma_t = (g_t)_* \sigma$
for some continuous curve $t \in I \mapsto g_t \in \Aut(Z,\id)$
starting from $g_0 = \id$
then we say that $\sigma$ and $\sigma'$ are \emph{isotopic},
and that $\{g_t\}_{t \in I}$ is an \emph{ambient isotopy}
that moves the section $\sigma$ to the section $\sigma'$.


Given $g \in \Aut(Z)$ and $\sigma \in \Sec(Z)$,
we define a new section $g_*\sigma \in \Sec(Z)$ by
\begin{equation}\label{e.push_section}
(g_* \sigma)(x) \coloneqq g(\sigma(f^{-1}(x))) \, ,
\end{equation}
where $f$ is the homeomorphism over which $g$ fibers.
A section $\sigma$ is called:
\begin{itemize}
\item \emph{$g$-invariant} if $g_* \sigma = \sigma$;
\item \emph{$\epsilon$-almost $g$-invariant} for some $\epsilon > 0$ if $\mathrm{d}(g_* \sigma, \sigma) < \epsilon$;
\item \emph{$g$-invariant up to homotopy} if $\sigma$ and $g_* \sigma$ are fibered homotopic;
\item \emph{$g$-invariant up to isotopy} if $\sigma$ and $g_* \sigma$ are isotopic.
\end{itemize}

We can now state the second main result of this paper:

\begin{theorem}\label{t.sections}
Let $f \colon X \to X$ be a minimal diffeomorphism of 
a compact manifold $X$ of positive dimension.
Let $Y$ be a compact connected manifold with finite fundamental group.
Consider a fiberwise smooth bundle $\bundle{Z}{X}{p}{Y}$
endowed with a fibered Riemannian structure,
and let $g \in \Aut(Z,f)$ be a fiberwise isometry.
Suppose that $\sigma \in \Sec(Z)$ is $g$-invariant up to isotopy.
Then for any $\epsilon>0$ there exists 
an  $\epsilon$-almost $g$-invariant section
$\omega \in \Sec(Z)$ that is fibered homotopic to $\sigma$.
\end{theorem}

\begin{remark}
Actually it is equivalent to suppose that $\sigma$ is $g$-invariant up to homotopy.
This equivalence can be proven by using vector bundle neighborhoods \cite[Thrm.~12.10]{Palais} (a tool generally used to endow $\Sec(Z)$ with a Banach manifold structure), but we will not provide the technical details.
\end{remark}

As a corollary of Theorem~\ref{t.sections},
we will show in \S~\ref{ss.coboundaries} 
that a cocycle on a compact Lie group with finite center
is an almost coboundary if and only if it is homotopic to a coboundary.

\medskip

Theorem~A from \cite{BN_geometric} also constructs almost invariant sections
for fiberwise isometries, but under hypotheses very different from those of Theorem~\ref{t.sections}.
See \S~\ref{ss.unification} for a discussion of
possible connections between these two results.

\subsection{Comments on the proofs and organization of the paper}

The broad strategy that we follow to prove Theorem~\ref{t.densedom}
is the same used in \cite{ABD1,ABD2} in a more restricted setting.
Absence of domination allows us 
to mix Lyapunov exponents and make the dynamics conformal, after a suitable perturbation.
Conformality allows us to induce certain fiberwise isometries,
and using almost invariant sections we introduce some weak domination with further perturbations.
(See \S\S~\ref{ss.dom_vs_conf}--\ref{ss.proof_densedom} for details.)

Let us comment how the construction of almost invariant sections
presented in this paper relates to previous ones.
The papers \cite{ABD1,ABD2} use ``dynamical stratifications'' (see \S~\ref{ss.strat})
and towers to construct almost-invariant sections.
Actually these constructions, which only form part of these papers,
can be considerably simplified by 
the geometric methods of \cite{BN_geometric} 
or the more specific linear-algebraic methods of \cite{BN_elementary}.
Unfortunately, these ``cleaner'' methods require a convenient geometry and  
do not apply to the situation considered here.
Thus our constructions are closer to those of \cite{ABD1,ABD2}
(though we do not directly use results from these papers).

The generality of Theorem~\ref{t.sections} creates new topological problems,
and we need two novel tools:
One tool is a certain regularity property of the dynamical stratifications
which, despite being natural, is not straightforward to obtain.
The other tool is actually not new,
but this is perhaps the first time it is used for dynamical applications:
it is a ``quantitative homotopy'' result from \cite{CS80,SW89}.

\medskip

The rest of this paper is organized as follows:
In Section~\ref{s.consequences} we explain a result from \cite{BN_elementary}
which is then combined with Theorem~\ref{t.sections}
to deduce Theorem~\ref{t.densedom}; we also explain an independent application of 
Theorem~\ref{t.sections} to almost coboundaries.
In Section~\ref{s.ingredients} we explain the two new tools 
mentioned above,
which we then employ in Section~\ref{s.proof_sections} to prove Theorem~\ref{t.sections}.
Section~\ref{s.examples} contains examples and remarks on the necessity of
the various hypotheses in our theorems, as well as questions for future research.
The more technical construction of regular dynamical stratifications
is given in Appendix~\ref{s.appendix}.


\medskip

\begin{ack}
Discussions with Carlos~Tomei (PUC--Rio) lead me to believe that 
slow homotopies should exist under mild conditions,
and so propelled me to search the literature 
until I found the papers by Calder, Siegel, and Williams.
\end{ack}

\section{Consequences of Theorem~\ref{t.sections}}\label{s.consequences}

\subsection{Domination versus conformality}\label{ss.dom_vs_conf}

The link between our two main Theorems \ref{t.densedom} and \ref{t.sections}
is made by means of the following result:

\begin{otherthm}[Bochi--Navas \cite{BN_elementary}]\label{t.BN}
Let $f \colon X \to X$ be a minimal homeomorphism of 
a compact space $X$ of finite dimension.
Let $\E$ be a vector bundle over $X$.
Then there exists a dense subset $\cI \subset \Aut(\E,f)$ such that for every $A \in \cI$,
\begin{enumerate}
\item\label{i.BN_noDS} 
either $A$ has a dominated splitting;
\item\label{i.BN_conformal}
or $A$ is conformal with respect to some Riemannian metric on $\E$.
\end{enumerate}
\end{otherthm}

Domination evidently fails in the second alternative, and it does so
in the most extreme of ways: all vectors in the same fiber are expanded 
(at time $1$) at exactly the same rate.

Let us summarize what is involved in the proof of this result.
The first part of the proof is to apply a theorem from \cite{B_Studia} 
which, extending previous results of \cite{BV,AB}, states that
generic elements of $\Aut(\E,f)$ either admit dominated splittings or are uniformly subexponentially quasiconformal (i.e., such that its Oseledets decompositions
are all trivial).
The one-phrase rationale behind it is this:
absence of domination allows Lyapunov exponents to be mixed by suitable perturbations.
The second part of the proof of Theorem~\ref{t.BN} is to 
construct a Riemannian metric with respect to which the 
quasiconformal distortion is small.
The third and final part is to perturb the automorphism to become
conformal with respect to this metric -- 
which is not obvious, because the new metric is usually very distorted 
as compared to the initial one.
The second and third parts of the proof can be carried out
by using elementary linear-algebraic tools, as it is done in \cite{BN_elementary},
or by geometric constructions on fiberwise isometries of suitable spaces,
as in \cite{BN_geometric}.
Finally, let us mention that 
the three parts of the proof can be refined in order to yield similar conformality 
properties inside the subbundles of the finest dominated splitting:
the result is Theorem~2.4 from \cite{BN_elementary},
and Theorem~\ref{t.BN} is actually a corollary of it.

\subsection{The Grassmannian bundle and deduction of Theorem~\ref{t.densedom}}\label{ss.proof_densedom}

Let us prepare the ground for the use of Theorem~\ref{t.sections}.
A general procedure for obtaining fiberwise isometries
is as follows:

\begin{proposition} \label{p.fibered_riem}
Consider a fiberwise smooth bundle $\bundle{Z}{X}{p}{Y}$ 
whose structural group $H \subset \Diff(Y)$ is compact.
Then the bundle admits a fibered Riemannian structure 
with respect to which any $H$-automorphism of $Z$ is a fiberwise isometry.
\end{proposition}

\begin{proof}
Start with any Riemannian metric on $Y$.
By averaging with respect to the Haar measure of $H$, we obtain a 
Riemannian metric on $Y$ that is preserved by $H$.
Therefore we can pull it back by bundle charts and obtain a well defined 
fibered Riemannian structure on $Z$.
This structure is obviously preserved by any $H$-automorphism of $Z$.
\end{proof}

\medskip

Given integers $1 \le k < m$, let $\Gr(k,m)$ denote the \emph{Grassmannian}
whose elements are the $k$-planes in $\R^m$;
these are compact connected manifolds, and their fundamental groups
are (see e.g.\ \cite[p.~189]{Arkowitz}):
\begin{equation}\label{e.pi1grass}
\pi_1 (\Gr(k,m)) = 
\begin{cases}
\Z   &\text{if $m=2$,} \\
\Z_2 &\text{if $m\ge 3$.}
\end{cases}
\end{equation}
Each linear automorphism of $\R^m$ induces a diffeomorphism of $\Gr(k,m)$ in the obvious way.
This defines a homomorphism 
$\iota \colon \GL(d,\R) \to \Diff(\Gr(k,m))$
whose kernel is formed by the nonzero multiples of the identity matrix.

If $\E$ is a $m$-plane bundle over a compact Hausdorff space $X$,
let $G_k(\E)$ denote the set of all $k$-planes contained in the fibers of $\E$.
This set can be given the structure of a fiber bundle 
with base space $X$, typical fiber $\Gr(k,m)$,
and structural group $\iota(\GL(d,\R))$.
Any automorphism $A$ of $\E$
induces an automorphism $\bar{A}$ of $G_k(\E)$.

Let us explain how a Riemannian metric on the vector bundle $\E$ 
induces a fibered Riemannian structure on the fiber $G_k(\E)$.
Given such a Riemannian metric, we use it to stiffen the fiber bundle structure of $\E$
so that the structural group is the orthogonal group $\O(m)$.
We also stiffen the fiber bundle $G_k(\E)$ so that the structural group is $\iota(\O(m))$.
Since this group is compact, Proposition~\ref{p.fibered_riem}
provides us with a fibered Riemannian structure on $G_k(\E)$ that has the following property: 
for any automorphism $A$ of $\E$ 
that is conformal with respect to the Riemannian metric,
the induced automorphism $\bar{A} \in \Aut(G_k(\E))$ is a fiberwise isometry.

\begin{proof}[Proof of Theorem~\ref{t.densedom}] 
Let $f \colon X \to X$ be a minimal diffeomorphism.
Let $\E$ be a $m$-plane bundle over $X$, where $m \ge 3$.
Fix a fibered-homotopy class $\cC \subset \Aut(\E,f)$.
Let $\cD$ be the open subset of $\cC$ formed by the automorphisms that have a dominated splitting.
Suppose that we are not in case~(\ref{i.obstruction}) in the statement of the theorem,
that is, there exist $A_0 \in \cC$ and $k \in \{1,2,\dots,m-1\}$
with a continuous invariant field of $k$-planes.
This means that the induced automorphism 
$\bar{A}_0 \colon G_k(\E) \to G_k(\E)$
has an invariant section $\sigma \colon X \to G_k(\E)$.

Take an arbitrary open set $\cV \subset \cC$;
we will show that $\cD \cap \cV \neq \emptyset$,
so concluding that property (\ref{i.densedom}) holds and therefore proving the theorem.

Let $\cI$ be the dense subset of $\Aut(\E,f)$ provided by Theorem~\ref{t.BN}.
Fix $A_1 \in \cI \cap \cV$.
If $A_1 \in \cD$ then we have nothing to show, so assume that $A_1 \not\in \cD$.
Then we are in case~(\ref{i.BN_conformal}) in Theorem~\ref{t.BN}, that is,
there is a Riemannian metric on $\E$ with respect to which $A_1$ is conformal.
As explained above, we can endow the bundle $G_k(\E)$ with a fibered Riemannian structure
with respect to which the automorphism $\bar{A}_1$ is a fiberwise isometry.

Since $A_0$ and $A_1$ belong to the class $\cC$,
there exists a fibered homotopy $(A_t)_{t \in I}$ in $\Aut(E,f)$ between $A_0$ and $A_1$.
Then $(\bar{A}_t \circ \bar{A}_0^{-1})_{t\in I}$ is an ambient isotopy 
that moves the section $\sigma$ to the section $(\bar{A}_1)_* \sigma$.
In particular, $\sigma$ is $\bar{A}_1$-invariant up to isotopy.

The fibers of the bundle $G_k(\E)$ satisfy the hypotheses of Theorem~\ref{t.sections}:
they are compact connected manifolds which by \eqref{e.pi1grass} 
have finite fundamental groups.
Therefore for each $i \in \N$
we can apply Theorem~\ref{t.sections} 
and obtain an $1/i$-almost $\bar{A}_1$-invariant
section $\omega_i \colon X \to G_k(\E)$.
This means that $\omega_i$ is uniformly $1/i$-close to the section $\omega_i'$
defined by $\omega_i'(x) \coloneqq A_{1\star}(f^{-1}x) (\omega_i(x))$.

For each $i \in \N$, we can find an automorphism $R_i \in \Aut(\E,\id)$ such that for each $x \in X$,
$R_i(x)$ is an orthogonal linear map (with respect to the Riemannian metric on the fiber $\E_x$)
and sends the $k$-plane $\omega_i(x)$ to the $k$-plane $\omega_i'(x)$.
Moreover, it is possible to choose the sequence $(R_i)$ converging to the identity autormorphism.

For each $i \in \N$ and $x \in X$, 
let $D_i(x)$ be the isomorphisms of $\E_x$ that preserves the $k$-plane $\omega_i(x)$ and its orthogonal complement $\omega_i^\perp(x)$,
and whose restriction to $\omega_i(x)$ (resp.\ $\omega_i^\perp(x)$) 
is $e^{1/i}$ (resp.\ $e^{-1/i}$) times the identity.
This defines a sequence of automorphisms $D_i \in \Aut(\E,\id)$ that converges to the identity.

The sequence $(B_i)$ on $\Aut(\E,f)$ defined by $B_i \coloneqq D_i \circ R_i \circ A_1$ converges to $A_1$.
Moreover each $B_i$ belongs to $\cD$, since it admits the dominated splitting $\omega_i \oplus \omega_i^\perp$.
Therefore $\cD \cap \cV \neq \emptyset$, as we wanted to prove.
\end{proof}

\subsection{Almost coboundaries}\label{ss.coboundaries}

We will describe another application of Theorem~\ref{t.sections}.

Let $f \colon X \to X$ be a homeomorphism of a compact Hausdorff space,
let $G$ be a topological group, and let $A \colon X \to G$ be continuous cocycle over $f$. We say that $A$ is a \emph{coboundary} if there exists a continuous map 
$B \colon X \to G$ such that 
$$
A(x) = B(f(x)) B(x)^{-1} \, .
$$
A cocycle is called an \emph{almost coboundary} if it is the limit of a sequence of coboundaries.

\begin{corollary}\label{c.coboundaries}
Let $f \colon X \to X$ be a minimal diffeomorphism of 
a compact manifold $X$ of positive dimension.
Let $G$ be a compact Lie group with finite center.
Then a cocycle $A \colon X \to G$ is an almost coboundary 
if and only if it is homotopic to a coboundary.
\end{corollary}

\begin{proof}
Endow $G$ with a bi-invariant metric,
and consider the product bundle $X \times G$ over $X$.
Then each cocycle $A \colon X \to G$ induces a fiberwise isometry 
$g(x,y) = (f(x), A(x) y)$;
moreover $A$ is a coboundary if and only if $g$ has an invariant section,
and $A$ is homotopic to a coboundary if and only if $g$ has a section
that is invariant up to isotopy.

First consider the case of connected $G$.
Since the Lie algebra of $G$ has trivial center,
by a theorem of Weyl (see e.g.\ \cite[p.~82]{Hsiang}),
the fundamental group of $G$ is finite.
Therefore the corollary follows from Theorem~\ref{t.sections}.

In general, if a cocycle $A$ is homotopic to a coboundary $x \mapsto B(f(x)) B(x)^{-1}$,
then the cocycle $x \mapsto B(f(x))^{-1} A(x) B(x)$ takes values in the identity component of $G$,
and therefore the corollary follows from the previous case.
\end{proof}

\section{Ingredients for the proof of Theorem~\ref{t.sections}} \label{s.ingredients}

\subsection{Speed control for homotopies}

We begin with an informal motivation for quantitative homotopy problems.
Let $X$ be any topological space, and let $Y$ be a Riemannian manifold.
A homotopy $F \colon X \times I \to Y$ is called \emph{$c$-Lipschitz} if 
$$
d(F(x,t), F(x,s)) \le c |t - s| , \quad \text{for all $t$, $s\in I$ and $x\in X$.}
$$
Assume that $X$ is compact;
then it is not difficult to see that any two homotopic maps 
$X \to Y$
are Lipschitz homotopic.
Calder, Siegel and Williams have dealt with this kind of question:
\emph{If $Y$ is compact, can we choose the homotopy above with a ``small'' Lipschitz constant?}
More precisely, \emph{is there a finite constant $b = b(X,Y)$ 
such that one can always find a $b$-Lipschitz homotopy 
between any given pair of homotopic maps $X \to Y$?}

The answer is clearly negative in general.
For example, it is easy to see that $b(I,S^1) = \infty$:
despite all maps $I \to S^1$ being homotopic, 
an homotopy that has to unwind many turns will necessarily have 
large Lipschitz constant.

Let us see some situations where the answer is positive.
First, $b(I,S^2)<\infty$; this can be shown by 
using a deformation retraction of the punctured sphere to a point.
The same trick shows that $b(I^n, S^{n+1}) < \infty$ for each $n$.
What about $b(I^2,S^2)$? 
Since any map $I^2 \to S^2$ can be lifted with respect to the Hopf fibration to a map $I^2 \to S^3$,
we can perform a controlled Lipschitz homotopy on $S^3$ and then project back to $S^2$.
Thus $b(I^2,S^2) < \infty$.
The same argument shows that $b(X,S^2) < \infty$ for any $2$-dimensional manifold $X$.

After considering the examples above, one may guess
that the topology of $X$ is not very important,
and it is the topology of $Y$ which determines the finiteness of $b(X,Y)$.
This is indeed true; in fact, the following very general result holds:

\begin{otherthm}[Calder--Siegel--Williams]\label{t.CSW}
Let $d \in \N$ and let $Y$ be a compact Riemannian manifold with finite fundamental group.
Then there exists $b = b(d,Y) > 0$ with the following properties.
Let $X$ be a compact CW-complex of dimension $d$, and let $A \subset X$ be a subcomplex.
Let $f_0$, $f_1 \colon X \to Y$ be homotopic relative to $A$.
There there exists a  $b$-Lipschitz homotopy relative to $A$ between the two maps.
\end{otherthm}

Theorem~\ref{t.CSW} is contained in Corollary~2.6 from \cite{SW89}.
In the case $A=\emptyset$, Theorem~\ref{t.CSW} was obtained previously in \cite{CS80}:
see Theorem~0.2 and Corollary~3.6 in that paper.

Let us give a brief sketch of the proofs,
which, similarly to the informal discussion above, 
involve lifting to a convenient larger space. 
For simplicity we discuss only the case $A = \emptyset$.
The pair of maps $f_0$, $f_1$ can be seen as a single map $(f_0,f_1) \colon X \to Y \times Y$.
Each homotopy between them corresponds to a lift of $(f_0,f_1)$
with respect to the fibration $p \colon C(I,Y) \to Y \times Y$ that sends a free path to its endpoints.
Finiteness of $\pi_1(Y)$ actually implies that the (infinite dimensional) fibers of $p$
have the homotopy type of a CW-complex with finitely many cells in each dimension.
Calder and Siegel use this property to deform maps 
$X \to C(I,Y)$ along the fibers of $p$ so the image becomes contained in 
a compact subset $C_d \subset C(I,Y)$ depending not on $X$, but only on its dimension $d$.
From compactness it is relatively simple to conclude the existence of a uniform Lipschitz constant.

\begin{remark}
If $F \colon X \times I \to Y$ is a given homotopy between $f_0$ and $f_1$ relative to $A$
then the $b$-Lipschitz homotopy $G \colon X \times I \to Y$ provided by Theorem~\ref{t.CSW}
is itself homotopic to $F$, relative to $A \times I \cup X \times \partial I$.
We will not use this fact, however.
\end{remark}

\subsection{Dynamical stratifications}\label{ss.strat}

In all this subsection, we assume that $X$ is a compact manifold of dimension $d > 0$, 
and $f \colon X \to X$ is a minimal diffeomorphism.
Our aim here is to describe certain decompositions of the space $X$ with 
good dynamical and topological properties.

Fix a compact set $K \subset X$ with nonempty interior.
Then, for each $x \in X$, let: 
\begin{align}
	\ell^+(x) &\coloneqq \min \{j \ge 0 ;\; f^j(x)    \in \interior K\} \, , \label{e.ell_plus}\\
	\ell^-(x) &\coloneqq \min \{j > 0   ;\; f^{-j}(x) \in \interior K\} \, , \label{e.ell_minus}\\
	\cL(x) &\coloneqq \{ j \in \Z ;\; -\ell^-(x) < j < \ell^+(x) , \ f^j(x) \in \partial K \} \, . \label{e.L}
\end{align}
By minimality, all these numbers are finite and uniformly bounded.
Following \cite[p.~75]{ABD2}, define closed sets: 
$$
X_i \coloneqq \{ x \in X ;\; \#\cL(x) \ge i\} \, ,
\quad i = 0,1,2,\dots
$$
The sequence 
$$
X = X_0 \supset X_1 \supset \cdots
$$
is called the \emph{dynamical stratification of $X$ associated to $K$.}
An example is pictured in Fig.~\ref{f.strat}.

\newcommand{\raio}  {.457}
\newcommand{\deltax}{.6571}
\newcommand{\deltay}{.2317}
\begin{figure}[htb]
	\begin{tikzpicture}[scale=6.5,font=\scriptsize]
		\begin{scope}
			\clip(-.5,-.5) rectangle (.5,.5); 
			
			\draw[blue, dashed](2*\deltax-2,2*\deltay-1) circle [radius=\raio]; 
			\draw[blue,very thick,domain= 29.03: 42.95] plot({2*\deltax-2+\raio*cos(\x)},{2*\deltay-1+\raio*sin(\x)});
				
			\draw[blue, dashed](2*\deltax-1,2*\deltay-1) circle [radius=\raio]; 
			\draw[blue,very thick,domain= 88.95:159.09] plot({2*\deltax-1+\raio*cos(\x)},{2*\deltay-1+\raio*sin(\x)});
			
			\draw[blue, dashed](2*\deltax-2,2*\deltay  ) circle [radius=\raio]; 

			\draw[blue, dashed](2*\deltax-1,2*\deltay  ) circle [radius=\raio]; 
			\draw[blue,very thick,domain=239.75:262.88] plot({2*\deltax-1+\raio*cos(\x)},{2*\deltay+\raio*sin(\x)});
			
			\draw[brown, dashed, name path=brown_SW](-\deltax  ,-\deltay  ) circle [radius=\raio]; 
			\draw[brown,very thick,domain= 59.75: 82.88] plot({-\deltax+\raio*cos(\x)},{-\deltay+\raio*sin(\x)});
			\draw[brown,very thick,domain=209.03:222.95] plot({-\deltax+\raio*cos(\x)},{-\deltay+\raio*sin(\x)});
			\draw[brown,very thick,domain=268.95:339.09] plot({-\deltax+\raio*cos(\x)},{-\deltay+\raio*sin(\x)});

			\draw[brown, dashed, name path=brown_SE](-\deltax+1,-\deltay  ) circle [radius=\raio]; 
			\draw[brown,very thick,domain= 59.75: 82.88] plot({-\deltax+1+\raio*cos(\x)},{-\deltay+\raio*sin(\x)});
			\draw[brown,very thick,domain=209.03:222.95] plot({-\deltax+1+\raio*cos(\x)},{-\deltay+\raio*sin(\x)});
			\draw[brown,very thick,domain=268.95:339.09] plot({-\deltax+1+\raio*cos(\x)},{-\deltay+\raio*sin(\x)});

			\draw[brown, dashed, name path=brown_NW](-\deltax  ,-\deltay+1) circle [radius=\raio]; 
			\draw[brown,very thick,domain= 59.75: 82.88] plot({-\deltax+\raio*cos(\x)},{-\deltay+1+\raio*sin(\x)});
			\draw[brown,very thick,domain=209.03:222.95] plot({-\deltax+\raio*cos(\x)},{-\deltay+1+\raio*sin(\x)});
			\draw[brown,very thick,domain=268.95:339.09] plot({-\deltax+\raio*cos(\x)},{-\deltay+1+\raio*sin(\x)});

			\draw[brown, dashed, name path=brown_NE](-\deltax+1,-\deltay+1) circle [radius=\raio]; 
			\draw[brown,very thick,domain= 59.75: 82.88] plot({-\deltax+1+\raio*cos(\x)},{-\deltay+1+\raio*sin(\x)});
			\draw[brown,very thick,domain=209.03:222.95] plot({-\deltax+1+\raio*cos(\x)},{-\deltay+1+\raio*sin(\x)});
			\draw[brown,very thick,domain=268.95:339.09] plot({-\deltax+1+\raio*cos(\x)},{-\deltay+1+\raio*sin(\x)});

			\draw[green!50!black, very thick, name path=green_SW](\deltax-1,\deltay-1) circle [radius=\raio]; 
			\draw[green!50!black, very thick, name path=green_SE](\deltax  ,\deltay-1) circle [radius=\raio]; 
			\draw[green!50!black, very thick, name path=green_NW](\deltax-1,\deltay  ) circle [radius=\raio]; 
			\draw[green!50!black, very thick, name path=green_NE](\deltax  ,\deltay  ) circle [radius=\raio]; 

			\draw[red, very thick, name path=red_C]( 0, 0) circle [radius=\raio];
			
			\fill[name intersections={of=red_C and green_SW, by={a,b}}] (a) circle (.2pt) (b) circle (.2pt);
			\fill[name intersections={of=red_C and green_NW, by={a,b}}] (a) circle (.2pt) (b) circle (.2pt);
			\fill[name intersections={of=red_C and green_NE, by={a,b}}] (a) circle (.2pt) (b) circle (.2pt);
			\fill[name intersections={of=red_C and brown_SW, by={a,b}}] (a) circle (.2pt) (b) circle (.2pt);
			\fill[name intersections={of=red_C and brown_NW, by={a,b}}] (a) circle (.2pt) (b) circle (.2pt);
			\fill[name intersections={of=red_C and brown_NE, by={a,b}}] (a) circle (.2pt) (b) circle (.2pt);
			\fill[name intersections={of=red_C and brown_SE, by={a,b}}] (a) circle (.2pt) (b) circle (.2pt);
			\foreach \x in {29.03,42.95}{\draw[fill] ({2*\deltax-2+\raio*cos(\x)},{2*\deltay-1+\raio*sin(\x)}) circle[radius=.2pt];}
			\foreach \x in {88.95,159.09}{\draw[fill] ({2*\deltax-1+\raio*cos(\x)},{2*\deltay-1+\raio*sin(\x)}) circle[radius=.2pt];}
			\foreach \x in {239.75,262.88}{\draw[fill] ({2*\deltax-1+\raio*cos(\x)},{2*\deltay+\raio*sin(\x)}) circle[radius=.2pt];}
				

		\end{scope} 

		\draw (-.5,-.5) rectangle (.5,.5); 
		


		\node[red, right] (zzero) at (.55,.12) {$\partial K$};
		\draw[thin, ->] (zzero) -- (.455,.12);
		
		\node[brown, right] (menosum) at (.55,.27) {$f^{-1}(\partial K)$};
		\draw[thin, ->] (menosum) -- (.51,.34);
		\draw[thin, ->] (menosum) -- (.51,.20);

		\node[blue, right] (dois) at (.55,-.04) {$f^{2}(\partial K)$};
		\draw[thin, ->] (dois) -- (.51,.04);
		\draw[thin, ->] (dois) -- (.51,-.12);

		\node[green!50!black, right] (um) at (.55,-.27) {$f(\partial K)$};
		\draw[thin, ->] (um) -- (.51,-.34);
		\draw[thin, ->] (um) -- (.51,-.20);

	\end{tikzpicture}
	\caption{An example of a dynamical stratification: $f$ is a translation of the torus $X = \R^2/\Z^2$ and $K$ is a disk such that $\bigcup_{j=0}^2 f^j(\interior K) = X$. The ``skeleton'' $X_1$ is formed by the thick lines, and is contained in the set $\bigcup_{j=-1}^2 f^j(\partial K)$, which is also pictured. The ``skeleton'' $X_2$ is formed by the thick dots.}\label{f.strat}
\end{figure}
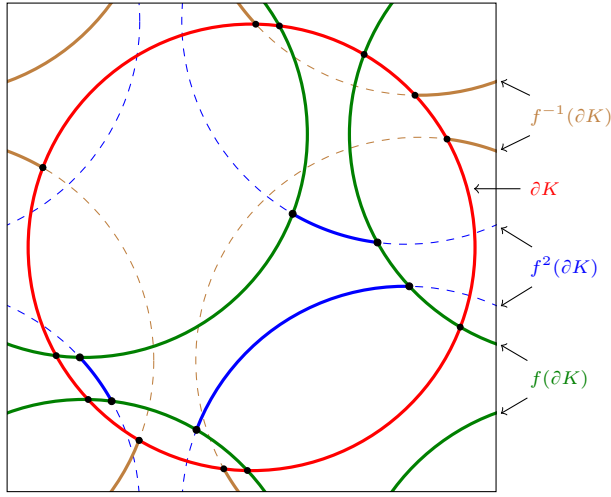

\begin{remark}
The usual definition of stratification also asks that each $X_i \setminus X_{i+1}$ is either empty or a submanifold of codimension $i$. We do not want to impose this requirement, although the stratifications we will actually use satisfy it.
\end{remark}

Recall that $I \coloneqq [0,1]$.
If $Y$ is a topological space and $A \subset Y$ is a closed set, 
then the pair $(Y,A)$ is said to have the \emph{homotopy extension property}
if $Y \times \{0\} \cup A \times I$ is a retract of $Y \times I$.
See either \cite[p.~25]{Arkowitz} or \cite[p.~15]{Hatcher} for 
a discussion of this property, including 
the equivalent characterization that explains its name.

\medskip

We say that a compact set $K\subset X$ of nonempty interior is \emph{regular},
or equivalently that the dynamical stratification $X_0 \supset X_1 \supset \cdots$
associated to it is \emph{regular} if the following two properties hold:
\begin{itemize}
\item $X_{d+1} = \emptyset$ (where $d \coloneqq \dim X$); 
\item for each $i\in \{0,1,\dots,d\}$, 
the pair $(X_i \cap K, X_{i+1} \cap K)$ has the homotopy extension property. 
\end{itemize}

It is not difficult to check that that the stratification of Fig.~\ref{f.strat}.
is regular. In general, regularity will be obtained by means of the following:

\begin{otherthm}\label{t.regularity}
Let $f \colon X \to X$ be a minimal diffeomorphism of 
a compact manifold $X$ of positive dimension.
Then every point in $X$ has a basis of 
neighborhoods consisting on regular embedded $d$-dimensional disks.
\end{otherthm}

The precise proof of the theorem is somewhat laborious,
and since its arguments are independent of the rest of the paper,
we present it separately in Appendix~\ref{s.appendix}.

\medskip

We will also need the following simple fact from \cite[p.~75]{ABD2},
whose proof we include for the reader's convenience:

\begin{lemma}\label{l.loc_const} 
The function $\ell^+$ is locally constant on each set $X_i \setminus X_{i+1}$.
\end{lemma}

\begin{proof} 
It is easy to see that $\cL$ and $\ell^+$ are upper semicontinuous on $X$,
that is, for any $x \in X$, if $y$ is sufficiently close to $x$ 
then $\cL(y) \subset \cL(x)$ and $\ell^+(y) \le \ell^+(x)$.
Assume for a contradiction that there exists
a sequence $(x_n)$ in $X_i \setminus X_{i+1}$ converging to some $x \in X_i \setminus X_{i+1}$
such that $\ell^+(x_n) < \ell^+(x)$ for each $n$.
By passing to a subsequence we can assume that $\ell^+(x_n) = k$ is independent of $n$.
On one hand, $f^k(x) \not\in \interior K$, and in the other hand
$f^k(x) = \lim f^k(x_n) \in \overline{\interior(K)} \subset K$,
showing that $f^k(x) \in \partial K$.
In particular, $k \in \cL(x) \setminus \cL(x_n)$ for each $n$.
However, both $\cL(x)$ and $\cL(x_n)$ have cardinality~$i$.
This contradiction proves the lemma.
\end{proof}

\section{Proof of Theorem~\ref{t.sections}} \label{s.proof_sections}

\subsection{Concentrating non-invariance}

Using regular dynamical stratifications,
we will construct sections that are invariant except on a small set.
These kind of sections were used in \cite{ABD1,ABD2}.
Here, for topological reasons, we need to keep track of not only one such section,
but a whole family of them, one for each automorphism along an isotopy.

\begin{lemma}\label{l.predom_inv}
Let $f \colon X \to X$ be a minimal diffeomorphism of 
a compact manifold $X$ of positive dimension.
Consider a fiberwise smooth bundle $\bundle{Z}{X}{p}{Y}$,
an automorphism $g \in \Aut(Z,f)$, 
a section $\sigma \in \Sec(Z)$ that is $g$-invariant up to isotopy, 
and a regular set $K \subset X$ contained in a trivializing domain of the bundle.
Then there exist:
\begin{itemize}
\item a continuous family of automorphisms $\{g_t\}_{t\in I} \subset \Aut(X, f)$
with $g_1 = g$;
\item a continuous family of sections $\{\phi_t\}_{t\in I} \subset \Sec(Z)$
with $\phi_0 = \sigma$;
\end{itemize}
such that:
\begin{equation}\label{e.predom_inv}
(x,t) \in (X \setminus \interior K) \times I \widecup X \times \{0\}
\ \Rightarrow \ g_t (\phi_t(x)) = \phi_t (f(x)) \, .
\end{equation}
\end{lemma}

\begin{proof}
Let $\{\hat{g}_t\}_{t \in I} \subset \Aut(X, \id)$ be an ambient isotopy
that moves the section $\sigma$ to the section $g_* \sigma$, that is,
$$
\hat{g}_0 = \id \quad \text{and} \quad (\hat{g}_1)_*(g_*(\sigma)) = \sigma \, .
$$
Define a continuous family $\{g_t\}_{t\in I}$ in $\Aut(X, f)$ by
$g_t \coloneqq \hat{g}_{1-t} \circ g$.
Then
$$
g_1 = g \quad \text{and} \quad (g_0)_*(\sigma) = \sigma \, .
$$

We want to define a map $\phi \colon X \times I \to Z$ such that
$\phi_t = \phi(\mathord{\cdot},t)$ are sections
satisfying \eqref{e.predom_inv}.
Since we also want $\phi_0 = \sigma$,
we start defining $\phi$ on $X \times \{0\}$ by $\phi(x,0) \coloneqq \sigma(x)$.
The extension to $X \times I$ will be made by a inductive procedure with $d+1$ steps, 
where $d \coloneqq \dim X$.
Consider the regular dynamical stratification associated to the set $K$:
$$
X = X_0 \supset X_1 \supset \cdots \supset X_{d+1} = \emptyset \, .
$$ 
Let $i \in \{0,\dots,d\}$ and assume
that $\phi$ is already continuously defined on $X \times \{0\} \widecup X_{i+1} \times I$,
and that it satisfies condition \eqref{e.predom_inv} where it makes sense.
(Notice that this assumption is already met for $i=d$, which is the starting point of the induction.)
We will explain how to extend $\phi$ to $X \times \{0\}\widecup X_i \times I$.

By regularity of the stratification, 
the pair $(X_i \cap K, X_{i+1} \cap K)$ has the homotopy extension property.
This means that there exists a retraction 
$$
r \colon (X_i\cap K) \times I \to L , \quad \text{where } 
L\coloneqq  (X_i\cap K) \times \{0\} \widecup (X_{i+1}\cap K) \times I \, .
$$
Notice that $L$ is the intersection of the domain of $r$
and the current domain of $\phi$.
Since $K$ is contained in a trivializing domain of the fiber bundle,
there exists a homeomorphism $h \colon K \times Y \to p^{-1}(K)$ 
such that $p \circ h$ equals the projection on the first factor.
Then there is a unique map 
$$
\eta \colon L  \to Y \quad \text{such that} \quad
\phi(x,t) = h(x,\eta(x,t)) \text{ for every } (x,t) \in L.
$$
We extend $\phi$ to an intermediate domain
\begin{equation}\label{e.intermediate}
X \times \{0\} \widecup \big(X_{i+1} \cup (X_i \cap \interior K) \big) \times I 
\end{equation}
by setting
$$
\phi(x,t) \coloneqq h \big( x ,  \eta \circ r(x,t)  \big) 
\quad \text{for } (x,t) \in (X_i \cap \interior K) \times I \, .
$$ 
Notice that this is coherent with the previously defined values of $\phi$,
and the new map $\phi$ is continuous.
For the new points in the domain of $\phi$, condition \eqref{e.predom_inv} is vacuously verified.

\medskip

To complete the induction step, we extend $\phi$ from the intermediate domain \eqref{e.intermediate} 
to $X \times \{0\} \cup X_i \times I$ by letting
$$
\phi(x,t) \coloneqq g_t^{-\ell^+(x)} \Big( \phi \big(  f^{\ell^+(x)}(x) \big) , t \Big) \quad \text{for every }
(x,t) \in (X_i \setminus X_{i+1}) \times I ;
$$ 
note that this map is well-defined because $\cL(f^{\ell^+(x)}(x))$ always has the same cardinality as $\cL(x)$,
and extends the previous $\phi$ because $\ell^+$ vanishes on $\interior K$ 
and $\phi_0$ is $g_0$-invariant.
This extension evidently keeps property \eqref{e.predom_inv} true where it makes sense.
Let us check that this new $\phi$ is continuous. 
It is sufficient to show that $\phi | X_i \times I$ is continuous.
Actually, 
it is  sufficient to prove continuity on points $(x,t) \in X_{i+1} \times I$,
since  $\ell^+$ is continuous on $X_i \setminus X_{i+1}$ by Lemma~\ref{l.loc_const}.

Take a sequence $(x_n,t_n)$ in $X_i \times I$ converging to $(x,t)$.
We can break the sequence $(x_n)$ into finitely many subsequences,
where each subsequence is either contained in $X_{i+1}$, 
or is contained in $X_i \setminus X_{i+1}$ and has a constant value of $\ell^+$.
Using that $\phi$ is continuous on the domain \eqref{e.intermediate},
it follows that $\phi(x_{n_k},t_{n_k}) \to \phi(x,t)$ for each of those subsequences $(x_{n_k})$.
Therefore the new map $\phi$ is continuous.

The induction stops after $d+1$ steps, 
when the map $\phi$ is defined on the set $X \times \{0\} \widecup X_0 \times I  = X \times I$.
The lemma is proved.
\end{proof}

\subsection{Dissipating non-invariance along a tower}

We have seen in Lemma~\ref{l.predom_inv} how to find sections
whose non-invariant part is concentrated in a small set $K$.
Next we want to ``dissipate'' this non-invariant part 
to a high tower $K \sqcup f(K) \sqcup \cdots \sqcup f^n(K)$
and in this way obtain an almost invariant section.
A major problem is that the sections we are working with may be extremely
twisted, and here is where the quantitative homotopy Theorem~\ref{t.CSW}
comes in handy.

\begin{proof}[Proof of Theorem~\ref{t.sections}]
Fix the diffeomorphism $f$,
the bundle $\bundle{Z}{X}{p}{Y}$, 
the automorphism $g$, 
and the section $\sigma$ 
satisfying the hypotheses of the theorem.
Let $\epsilon>0$ be arbitrary.

We endow $Y$ with a Riemannian structure.
Let $d = \dim X$, and let $b = b(d,Y)$ be given by Theorem~\ref{t.CSW}.

Let $U_0 \subset X$ be a trivializing domain for the fiber bundle.
Then there exists a homeomorphism $h \colon U_0 \times Y \to p^{-1}(U_0)$
such that for each $x \in U_0$, 
the map $h (x, \mathord{\cdot})$ is a diffeomorphism
from $Y$ to the fiber $Z_x = p^{-1}(x)$.
Take an open $U \neq \emptyset$ such that $\overline{U} \subset U_0$.
Let $c>0$ be an upper bound 
for the Lipschitz constants of all maps 
$h (x, \mathord{\cdot})$ with $x \in U$. 
Fix an integer 
$$
n > \frac{c b}{\epsilon} \, .
$$ 

By Theorem~\ref{t.regularity},
we can choose a regular embedded closed $d$-dimensional disk $K \subset U$
sufficiently small so that it is disjoint from its $n$ first iterates.
We apply Lemma~\ref{l.predom_inv} to this set $K$,
and so obtain continuous 
families $\{g_t\}_{t\in I} \subset \Aut(X, f)$ and $\{\phi_t\}_{t\in I} \subset \Sec(Z)$
such that $g_1 = g$, $\phi_0 = \sigma$, and 
\begin{equation}\label{e.predom_inv_again}
(x,t) \in (X \setminus \interior K) \times I \widecup X \times \{0\}
\ \Rightarrow \ g_t (\phi_t(x)) = \phi_t (f(x)) \, .
\end{equation}
It follows that the sections
$((g^{-n}_t)_* \phi_t)(x) = g_t^{-n}(\phi_t(f^n(x)))$
satisfy:
\begin{equation}\label{e.tower}
((g^{-n}_t)_* \phi_t)(x) = \phi_t(x) \quad \text{if} \quad x \not\in\bigcup_{j=0}^{n-1} f^{-j}(\interior K)
\quad \text{or} \quad t = 0.
\end{equation}

\begin{lemma}\label{l.cube}
The restrictions of the sections $\phi_1$ and $g^{-n}_* \phi_1$ to the disk $K$ are fibered homotopic 
relative to $\partial K$.
In other words, there exists a continuous map $\eta \colon K \times I \to Y$ such that
$$
\phi_1(x)            = h(x, \eta(x,0)) , \quad 
(g^{-n}_* \phi_1)(x) = h(x, \eta(x,1))   \quad \text{for all } x \in K , 
$$
and moreover if $x \in \partial K$ then $\eta(x,s)$ does not depend on $s \in I$.
\end{lemma}

\begin{proof}
The pair $(K,\partial K)$ is homeomorphic to $(D^d, S^{d-1})$ -- the unit disk and the unit sphere in $\R^d$.
Therefore the pair 
$$
\big( K \times I, \  \partial(K\times I) \big) =  
\big( K \times I, \  \partial K \times I \widecup K \times\{0,1\} \big) 
$$ 
is homeomorphic to $(D^{d+1}, S^d)$.
It is well-known that the latter pair has the homotopy extension property
(see e.g.\ \cite[p.~9]{Arkowitz} or \cite[p.~15]{Hatcher}),
and thus so has the former.
This means that there exists a retraction $r \colon K \times I \times I \to C$, where
$$
C \enspace{\coloneqq}\enspace K \times I \times \{0\} \widecup \partial K \times I \times I \widecup K \times\{0,1\} \times I \, .
$$
Define continuous maps $\xi_0$, $\xi_1 \colon K \times I \to Y$ by:
$$
\phi_t(x)               = h(x, \xi_0(x,t)) \, , \quad 
((g^{-n}_t)_* \phi_t)(x) = h(x, \xi_1(x,t)) \, . 
$$
It follows from \eqref{e.tower} that
$$
\xi_0 = \xi_1 \quad \text{on} \quad K \times \{0\} \widecup \partial K \times I \, .
$$
Define a map $\xi \colon C \to Y$ by
$$
\xi(x,s,t) \coloneqq
\begin{cases}
	\xi_0(x,t) &\quad\text{if } (x,t) \in K \times \{0\} \widecup \partial K \times I \, , \\
	\xi_s(x,t) &\quad\text{if } s\in \{0,1\} \, .
\end{cases}
$$
Notice that $\xi$ is well-defined and continuous.
Extend $\xi$ to $K \times I \times I$ by imposing $\xi = \xi \circ r$.
The announced map $\eta$ is $\eta(x,s) \coloneqq \xi(x,s,1)$.
\end{proof}

Applying the Calder--Siegel--Williams Theorem~\ref{t.CSW}, 
we find another homotopy $\zeta \colon K \times I \to Y$
relative to $\partial K$ between 
the maps $\eta(\mathord{\cdot},0)$ and $\eta(\mathord{\cdot},1)$ 
with the additional property of being $b$-Lipschitz.

The desired section $\omega \colon X \to Z$ is defined as follows: 
$$
\omega(x) \coloneqq
\begin{cases}
g^j \left( h \left( f^{-j}(x), \zeta \left( f^{-j}(x), \frac{j}{n} \right) \right)\right)
&\text{if } x \in f^j(\interior K), \ 0 \le j \le n. \\
\phi_1(x) 
&\text{if } x \not\in\bigcup_{j=0}^{n} f^j(\interior K) \, .
\end{cases}
$$
Continuity follows from the fact that the homotopy $\zeta$ is relative to $\partial K$.
Notice that $\omega$ also coincides with $\phi_1$ on $K \cup f^n(K)$.

Let us check almost-invariance of $\omega$.
If $x \not\in\bigcup_{j=0}^{n-1} f^j(\interior K)$
then by \eqref{e.predom_inv_again}, $\phi_1(f(x)) = g(\phi_1(x))$,
that is, 
$$
\omega(f(x)) = g(\omega(x)).
$$
If, on the other hand, $x \in f^j(\interior K)$ for some $j \in \{0,1,\dots,n-1\}$ then
\begin{align*}
\omega(f(x))   &\coloneqq g^{j+1} \left( h \left( f^{-j}(x), \zeta \left( f^{-j}(x), \tfrac{j+1}{n} \right) \right)\right) \, ,
\\
g(\omega(x)) &\coloneqq g^{j+1} \left( h \left( f^{-j}(x), \zeta \left( f^{-j}(x), \tfrac{j}{n} \right) \right)\right) \, .
\end{align*}
Using that the maps $\zeta(f^{-j}(x), \mathord{\cdot})$ and
$h(f^{-j}(x) , \mathord{\cdot})$ 
are respectively $b$-Lipschitz and $c$-Lipschitz,
and that the map $g^{j+1}$ is isometric along fibers of $Z$,
it follows that
$$
\mathrm{d}_x \big( \omega(f(x)) , g(\omega(x)) \big) \le 
\frac{cb}{n} < \epsilon \, .
$$
This shows that the section $\omega$ is $\epsilon$-almost invariant under $g$.

Finally, let us check that $\omega$ and $\sigma = \phi_0$ are fibered homotopic.
Notice that if we replace the fraction $j/n$ that appears in the definition of $\omega$ by $1$ 
we obtain the following section, which is fibered homotopic to $\omega$:
$$
\tilde\omega(x) \coloneqq
\begin{cases}
(g_*^{-(n-j)} \phi_1)(x)
&\text{if } x \in f^j(\interior K), \ 0 \le j \le n. \\
\phi_1(x) 
&\text{if } x \not\in\bigcup_{j=0}^{n} f^j(\interior K) \, .
\end{cases}
$$
For each $j$ with $0 \le j \le n$,
the restrictions of the sections $\phi_1$ and $g^{-(n-j)}_* \phi_1$ to the disk $f^j(K)$ 
are fibered homotopic relative to $\partial f^j(K)$,
by the exact same argument of the proof of Lemma~\ref{l.cube}.
It follows that the sections $\tilde \omega$ and $\phi_1$ are fibered homotopic.
Since the latter is obviously fibered homotopic to $\phi_0 = \sigma$,
we conclude that $\omega$ and $\sigma$ are fibered homotopic, as announced.
This ends the proof of Theorem~\ref{t.sections}.
\end{proof}

\section{Further comments and questions}\label{s.examples}

In this section we collect a number of examples 
that illustrate and test the sharpness of Theorems~\ref{t.densedom} and \ref{t.sections}.
Some questions are posed along the way.

\subsection{A space of cocycles containing classes of both types}\label{ss.two_types}

We will give an example where $\Aut(\E,f)$ contains fibered homotopy classes of both types 
(\ref{i.obstruction}) and (\ref{i.densedom}) in Theorem~\ref{t.densedom}, and therefore such that domination is neither empty nor dense. 

Since the sphere $S^3$ is a Lie group containing circle subgroups,
by a theorem of Fathi and Herman \cite{FH},
there exists a minimal diffeomorphism $f \colon S^3 \to S^3$ that is homotopic to the identity. 
Let $\E$ be the trivial bundle $S^3 \times \R^3$. 
Then $\Aut(\E,f) = C(S^3, \GL(3,\R))$.

We regard the sphere $S^3$ as the group of unit quaternions in $\mathbb{H} = \R^4$,
the space $\R^3$ as the set of purely imaginary quaternions, and $S^2$ as $S^3 \cap \R^3$.
Let $\rho \colon S^3 \to \SO(3)$ be the homomorphism (and also a double covering)
that associates to each unit quaternion $q \in S^3$ the orthogonal linear map $v \in \R^3 \mapsto q^{-1}vq \in \R^3$ (see \cite[p.~105]{Thurston}, \cite[p.~75]{GHV}).
For each $n \in \Z$, let $P_n \colon S^3 \to S^3$ be the power map $q \mapsto q^n$,
and let $A_n$ be the composition of the following maps:
$$
S^3 \xrightarrow{P_n} S^3 \xrightarrow{\rho} \SO(3) \hookrightarrow \GL(3,\R) \, .
$$
Let $\cC_n$ be the fibered homotopy class of $A_n$.

\begin{fact}\label{fa.two_types}
The class $\cC_n$ is of type (\ref{i.obstruction}) if $n \neq 0$, 
and of type (\ref{i.densedom}) if $n = 0$.
\end{fact}

Before proving this, we need to establish the following:

\begin{fact}\label{fa.Hopf}
For each continuous map $h \colon S^3 \to S^2$ there exists a unique $n \in \Z$ such that
$h$  is homotopic to the map $h_n(q) \coloneqq (A_n(q))(\mathbf{i})$,
where $\mathbf{i} = (0,1,0,0)$ is the first imaginary unit in quaternionic space $\mathbb{H} = \R^4$.
\end{fact}

\begin{proof}
Since the Hopf invariant $H \colon \pi_2(S^3) \to \Z$ is an isomorphism 
it is sufficient to show that each map $h_n$ has Hopf invariant $H(h_n) = n$.
Since $h_1$ is the projection map of the Hopf fibration (see \cite[p.~106]{Thurston}),
its Hopf invariant is $1$.
It follows that the Hopf invariant of $h_n = h_1 \circ P_n$ 
is the degree of the map $P_n$ (see \cite[p.~428]{Hatcher}),
which is exactly $n$ (see \cite[p.~104]{GHV}).
\end{proof}

\begin{proof}[Proof of Fact~\ref{fa.two_types}]
Since $\cC_0$ is the class of cocycles homotopic to constant, it is indeed of type (\ref{i.densedom}).
Conversely, suppose that the class $\cC_n$ is of type (\ref{i.densedom}), that is,
some $A \in \cC_n$ has a continuous invariant $k$-plane field $\sigma \colon S^3 \to \Gr(k,3)$ for some $k \in \{1,2\}$. We will only discuss the case $k=1$, since the case $k=2$ is entirely analogous. Since $S^3$ is simply connected, the line field $\sigma$ is orientable,
that is, it lifts to a map $h \colon S^3 \to S^2$. Since $f \simeq \id$ and $A \simeq A_n$,
the map $h$ is homotopic to the map $h'(q) \coloneqq (A_n(q))(h(q)) = q^{-n} h(q) q^n$.
By Fact~\ref{fa.Hopf}, there exists $m \in \Z$ be such that $h \simeq h_m$.
Then $h' \simeq h_{m+n}$, and by Fact~\ref{fa.Hopf} again we have $n = 0$.
\end{proof}

\subsection{Different types of domination inside the same class}\label{ss.indices}

The \emph{indices} of a dominated splitting $\E = \E^1 \oplus \E^2 \oplus \cdots \oplus \E^k$
(where $\E^i$ dominates $\E^{i+1}$) are the dimensions of the subbundles
$\E^1$, $\E^1\oplus \E^2$, \dots, $\E^1 \oplus \cdots \oplus \E^{k-1}$.

Let us show that if $\cC$ is a fibered homotopy class
of type (\ref{i.densedom}) in Theorem~\ref{t.densedom},
it is not necessarily true that there is a common index of domination that appears open and densely in $\cC$.
	
Let $X = S^3 \times S^2 \times S^6$. 
Let $\E'$, $\E''$, and $\E$ be the vector bundles with base space $X$ 
whose fibers over $(x_1,x_2,x_3) \in S^3 \times S^2 \times S^6$
are respectively $\E'(x) \coloneqq T_{x_2} S^2$, $\E''(x) \coloneqq T_{x_3} S^6$,
and $\E(x) \coloneqq \E'(x) \oplus \E''(x)$. 

\begin{fact}
$\E$ has no subbundle of fiber dimension $4$.
\end{fact}

\begin{proof}
Let $\F$ be a subbundle of fiber dimension $4$.
Since $X$ is simply connected, $\F$ is orientable.
Then $\E = \F \oplus \F^\perp$ (oriented Whitney sum).
Let us consider cohomology groups with integer coefficients.
Since $H^m(S^n) \neq 0$ iff $m = 0$ or $n$,
and $4$ is not a sum of different numbers in $\{3,2,6\}$,
by the K\"unneth formula the cohomology group $H^4(X)$ vanishes.
In particular, the Euler classes $e(\F)$, $e(\F^\perp)$ vanish,
and therefore so does  
$e(\E) = e(\F) \smallsmile e(\F^\perp)$ (see \cite[p.~100]{MS}).
On the other hand, the vector bundle 
$\bundle{\E}{X}{}{\R^8}$ is the cartesian product (see \cite[p.~27]{MS})
of the following three vector bundles 
$$
\bundle{\G}{S^3}{}{\{0\}}, \quad
\bundle{TS^2}{S^2}{}{\R^2}, \quad
\bundle{TS^6}{S^6}{}{\R^6}, \quad 
$$
Therefore (see \cite[p.~100]{MS}),
$$
e(\E) = \underbrace{e(\G)}_{\in H^0(S^3)=\Z} \times \underbrace{e(TS^2)}_{\in H^2(S^2)=\Z} 
\times \underbrace{e(TS^6)}_{\in H^6(S^6)=\Z} \, ,
$$
where $\times$ is the cross product.
But $e(\G) = 1$ and 
$e(TS^{2k}) = \chi(S^k) = 2$, 
which
imply that $e(\E) \neq 0$.  
Contradiction.
\end{proof}

\begin{fact}
There is a minimal homeomorphism $f \colon X \to X$
and there is a homotopy class $\cC \subset \Aut(\E,f)$ 
containing an automorphism having a dominated splitting with index $2$ 
and another automorphism having a dominated splitting with index $6$.
\end{fact}

\begin{proof}[Proof of the fact]
Consider a free action of $S^1$ on $S^3$ by diffeomorphisms.
By multiplying by the identity on $S^2 \times S^6$ we obtain a free action  
$t \in S^1 \mapsto \phi_t \in \Diff(X)$.
Choose some $\phi_t \neq \id_X$ and call it $g$.

For each $\lambda > 0$, let $B_\lambda \in \Aut(\E, g)$ be the automorphism
that preserves the subbundles $\E'$ and $\E''$, 
whose restriction to $\E'$ is an homothecy of factor $\lambda$,
and whose restriction to $\E''$ is an homothecy of factor $\lambda^{-1}$.

Since $\{\phi_t\}$ is a free action of the circle,
by a theorem of Fathi and Herman \cite{FH}, 
minimal diffeomorphisms form a residual subset of 
the $C^\infty$-closure of the union of the $C^\infty$-conjugacy classes of
the maps $\phi_t$.
In particular, there is a sequence of diffeomorphisms $h_n \to \id_X$
such that each $f_n = h_n \circ g$ is minimal. 
Taking the partial derivative of the $S^2 \times S^6$ component
of $h_n$ with respect to itself, 
we obtain a sequence of automorphisms $D_n \in \Aut(\E, h_n)$ that converges to $\id_\E$.

Let $A_{\lambda,n} \coloneqq D_n \circ B_\lambda \in \Aut(\E, f_n)$.
For each $t$, we have $A_{\lambda,n} \to B_\lambda$ as $n \to \infty$.
Since $B_2$ and $B_{1/2}$ have dominated splittings of respective indices $2$ and $6$,
if $n$ is large enough then $A_{2,n}$ and $A_{1/2,n}$ also have dominated splittings of indices $2$ and $6$.
These two automorphisms are homotopic in $\Aut(\E, f_n)$.
So, taking $f = f_n$, the fact is proved.
\end{proof}

However, no element of $\cC$ can have simultaneously domination indices $2$ and~$6$,
because it would then have an invariant ``middle'' bundle of dimension $4$,
which we have seem that is impossible.

\subsection{Failure of Theorem~\ref{t.densedom} in dimension two}\label{ss.dim2}

The statement of Theorem~\ref{t.densedom} is false for $m=2$, even for trivial bundles.
For example, if $f$ is minimal but not uniquely ergodic
and $A$ is a cocycle of rotations homotopic to constant
whose fibered rotation numbers are not the same 
for all invariant measures then $A$ cannot be approximated by dominated cocycles.
For similar reasons, the theorem also fails for uniquely ergodic homeomorphisms
such that the range of the Schwartzman asymptotic cycle is not dense.
See \cite{ABD2} for details.

For $\SL(2,\R)$-cocycles (or, slightly more generally, 
orientation-preserving linear cocycles), 
these are basically all the possible counterexamples,
as it follows from the results of \cite{ABD2}.

In the lack of orientability or when the vector bundle is nontrivial,
new topological and dynamical problems appear.
We hope to address these in a later paper. 

\medskip

Going to vector bundles $\E$ of arbitrary fiber dimension,
another interesting problem 
is to describe the domination types that appear openly inside 
a given fibered homotopy class in $\Aut(\E,f)$.

\subsection{Theorem~\ref{t.sections} fails if $g$ is not a fiberwise isometry}

Consider the product bundle $\bundle{S^3 \times S^2}{S^3}{}{S^2}$.
Its sections can be identified with continuous maps $S^3 \to S^2$.
As in \S~\ref{ss.two_types}, consider a minimal diffeomorphism $f \colon S^3 \to S^3$
homotopic to the identity.
Define $g \in \Aut(Z,f)$ by $g(x,y) \coloneqq (f(x),h(y))$, where $h \colon S^2 \to S^2$
is a diffeomorphism homotopic to the identity 
with a single attracting fixed point and a single repelling fixed point
(e.g., a ``north pole -- south pole'' map).
Then any section is $g$-invariant up to homotopy.
On the other hand, if $\epsilon>0$ is sufficiently small, then 
every $\epsilon$-almost $g$-invariant section is homotopic to a constant.
Since there exist sections that are not homotopic to constants
(for example, the Hopf map $S^3 \to S^2$),
we conclude that Theorem~\ref{t.sections} does not apply to $g$.

\subsection{More general fibers}\label{ss.unification}

For a fiberwise isometry $g$ of a bundle having simply connected 
fibers of nonpositive curvature,
Theorem~A (with Remark 2.15) from \cite{BN_geometric} characterizes the value
$$
\inf \big\{\epsilon>0 ; \; \text{$g$ has an $\epsilon$-invariant section}\}
$$
as the \emph{maximal drift} of the cocycle, defined as the linear rate of
growth of the distances between the iterates of an arbitrary section and itself.

If we want to extend such a result to bundles whose fibers are non contractible,
it seems natural to measure distances between sections by using the 
\emph{homotopy distance} \cite{CS80} instead.
If the fiber $Y$ is compact with finite fundamental group
then the Calder--Siegel theorem states that these homotopy distances are uniformly bounded,
and therefore the ``maximal homotopy drift'' would always vanish when it is finite.

These remarks indicate that Theorem~A from \cite{BN_geometric}
and Theorem~\ref{t.sections} may be manifestations of a more general phenomenon.

\appendix
\section{Construction of regular dynamical stratifications}\label{s.appendix}

In this appendix we prove Theorem~\ref{t.regularity}.
The proof has basically three steps:
First we define a transversality property concerning the iterates of the boundary $\partial K$
of an embedded $d$-dimensional disk $K$, 
and show that this property can always be obtained by perturbation.
Second, we show that this transversality property implies the regularity of 
a certain auxiliary stratification.
Third, we deduce the regularity of the dynamical stratification. 

Before going into the details, let us highlight the basic ideas.
To check that a pair $(Y,A)$ has the homotopy extension property 
we only need to understand how a neighborhood of $A$ fits inside $Y$:
see \cite[Example~0.15, p.~15]{Hatcher}.
In our situation, this local topology is controlled using transversality
between $\partial K$ and its iterates.
If the dimension $d$ equals $2$, a typical situation is shown in Fig.~\ref{f.strat}.
In dimension $3$, a typical neighborhood of a point in $X_3 \cap \interior K$
is shown in Fig.~\ref{f.d3}.
A direct construction of the necessary retractions in order to 
prove regularity of dynamical stratifications would be messy,
so we use as a convenient technical device some auxiliary stratifications 
with a simpler local topology.

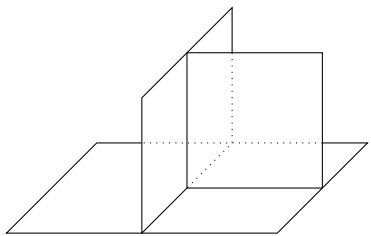
\begin{figure}[htb]
	\begin{tikzpicture}[scale=.6]
		\draw (3,2) -- (2,2) -- (0,0) -- (6,0) -- (8,2) -- (7,2); 
		\draw[dotted] (7,2) -- (3,2);                             
		\draw (4,1) -- (3,0) -- (3,3) -- (5,5) -- (5,4);          
		\draw[dotted] (5,4) -- (5,2) -- (4,1);                    
		\draw (4,1) -- (4,4) -- (7,4) -- (7,1) -- cycle;          
	\end{tikzpicture}
	\caption{The skeleton $X_1$ of a regular dynamical stratification in dimension $d=3$ around a point in $X_3 \cap \interior K$.}
	\label{f.d3}
\end{figure}

\subsection{Transverse hits}\label{ss.transv}

Let $E$ be a finite-dimensional real vector space.
Let $H_1$, \dots, $H_k$ be a finite family of hyperplanes of $E$ (i.e., codimension $1$ subspaces).
Let $\lambda_1$, \dots, $\lambda_k \in E^*$ be linear functionals such that $\Ker \lambda_i = H_i$
for each $i$; they are unique up to nonzero factors.
We say that the hyperplanes $H_1$, \dots, $H_k$ are \emph{independent} 
if the linear functionals $\lambda_1$, \dots, $\lambda_k$ are linearly independent.
An empty family of hyperplanes is also considered independent.

\begin{lemma}\label{l.transv_equivalence}
A family of hyperplanes $H_1$, \dots, $H_k$ of $E$
is independent iff 
their cartesian product $H_1 \times \cdots \times H_k$ is transverse 
to the diagonal $\Delta_k(E)$ of the cartesian power $E^k$.
\end{lemma}

\begin{proof} 
For each hyperplane $H_i$, fix a functional $\lambda_i \in E^*$ whose kernel is $H_i$.

Assume that the functionals $\lambda_1$, \dots, $\lambda_k$ are linearly independent.
Take vectors $u_1$, \dots, $u_k \in E$ such that $\lambda_i(u_j) = \delta_{ij}$.
Given $(v_1, \dots, v_k) \in E^k$, let $w = \sum_{i=1}^k \lambda_i(v_i) u_i$.
Then $v_j - w \in H_j$ for each $j$, which shows that $(v_1, \dots, v_k)$
is spanned by $H_1 \times \cdots \times H_k$ and $\Delta_k(E)$.
This proves that these two spaces are transverse.

Conversely, assume that $H_1 \times \cdots \times H_k \transverse \Delta_k(E)$.
Suppose $a_1$, \dots, $a_k \in \R$ are such that $\sum_i a_i \lambda_i = 0$.
Then the linear map $\Lambda \colon E^k \to \R$ defined by 
$\Lambda(v_1, \dots, v_k) \coloneqq \sum_{i=1}^k a_i \lambda_i(v_i)$.
vanishes both on the product $H_1 \times \cdots \times H_k$ and on the diagonal $\Delta_k E$,
and so it must be zero. Hence $a_i = 0$ for every $i$, showing that
$\lambda_1$, \dots, $\lambda_k$ are linearly independent.
\end{proof}

Let $K \subset X$ be a embedded $d$-dimensional (closed) disk.
Let $N$ be a finite set of integers, and let $x \in X$ be a point.
We say that \emph{the $N$-hits of $x$ at $\partial K$ are transverse} 
if the hyperplanes $Df^{-j} \left(T_{f^j(x)} (\partial K) \right) \subset T_x X$,
where $j$ runs on the elements of $N$ such that $f^j(x) \in \partial K$,
form an independent family.
If this condition is satisfied for every $x\in X$
then we say that $K$ has the \emph{transverse $N$-hits property}.

\begin{lemma}\label{l.transverse_hits}
Let $K \subset X$ be an embedded $d$-dimensional disk,
and let $U$ be a neighborhood of $\partial K$.
Then, for any finite set $N \subset \Z$,
there exists an embedded $d$-dimensional disk $\tilde K$
with the transverse $N$-hits property
and such that $\tilde K \mathbin{\vartriangle} K \subset U$.
\end{lemma}

\begin{proof}
Let $K \subset X$ be the image of an embedding
$h$ of the closed unit disk $\bar{B}(0,1) \subset \R^d$.
We can extend $h$ to a diffeomorphism between 
the open disk $B(0,2)$ and a neighborhood of $K$.
Let $U$ be any given neighborhood of $\partial K$.
Reducing $U$ if necessary, we can assume it is 
the image under $h$ of a spherical shell $B(0,1+\delta) \setminus \bar{B}(0,1-\delta)$,
for some $\delta \in (0,1)$.

Fix a finite set $N \subset \Z$, and let $n$ be its diameter.
Since $f$ has no periodic points, 
we can cover the unit sphere $S^{d-1}$ by open disks $B_1$, \dots, $B_m$ on $\R^d$
of radii less than $1/2$ and such that 
\begin{equation}\label{e.disjoint_images}
h(B_\ell) \cap f^j(h(B_\ell)) = \emptyset \quad \text{for all $j\in\{1,\dots,n\}$ and $\ell\in \{1,\dots,m\}$.}
\end{equation}
Let $\{\rho_1, \dots, \rho_m\}$ be a $C^\infty$ partition of unity subordinate to this cover.

For each $k \in \{1,2,\dots,n\}$, define a map 
$$
\Psi_k \colon (B(0,\epsilon))^m \times (S^{d-1})^{k} \to X^{k} 
$$
(where $\epsilon \in (0,\delta)$ will be determined later) by
$$
\Psi_k (y_1, \dots, y_m, z_1, \dots, z_k) \coloneqq 
\left( h \left( z_i + \sum_{\ell=1}^m \rho_\ell(z_i) y_\ell \right) \right)_{i=1,\dots,k}
$$
For each subset $J \subset N$ of cardinality $k$, 
say $J = \{j_1 < \cdots < j_k\}$, let
$f_J \colon X^{k} \to X^{k}$ be the diffeomorphism
$$
f_J(x_1,\dots,x_k) = \big(f^{j_1}(x_1), f^{j_2}(x_2), \dots, f^{j_k}(x_k) \big)  \, .
$$
Let $\Delta_{k} X$ be the diagonal of $X^{k}$,
and let $G_J \coloneqq f_J(\Delta_{k} X)$;
both are closed submanifolds of $X^{k}$.

\begin{fact}\label{fa.transv}
If $\epsilon \in (0,\delta)$ is chosen sufficiently small then 
for every nonempty $J \subset N$, if $k \coloneqq \# J$ 
then the map $\Psi_k$ is transverse to the submanifold $G_J$.
\end{fact}

\begin{proof}
Let $J = \{j_1 < \cdots < j_k\} \subset N$ and 
let $(z_1, \dots, z_k) \in (S^{d-1})^{k}$ be such that 
$\Psi_k (0, \dots, 0, z_1, \dots, z_k) \in G_J$,
that is, there exists $x_1 \in \partial K$ such that
\begin{equation}\label{e.hit}
h(z_i) = f^{j_i}(x_1) \quad \text{for every $i \in \{1,\dots, k\}$.}
\end{equation}
We will actually prove that the derivative of $D\Psi_k$
at the point $(0, \dots, 0, z_1, \dots, z_k)$ is onto, which implies the fact.

Indeed, for each $i \in \{1,\dots,k\}$, 
we can choose $\ell_i$ such that $\rho_{\ell_i}(z_i) \neq 0$,
and in particular $z_i \in B_{\ell_i}$.
By \eqref{e.hit}, the points $h(z_i)$ belong to a common segment of orbit of $f$ of length at most $n+1$,
and so using \eqref{e.disjoint_images} we conclude that the indices $\ell_i$ are pairwise distinct.
In particular, $\rho_{\ell_i}(z_t) = 0$ whenever $t \neq i$.
Suppose that 
$(y_1, \dots, y_m)$ is such that $y_\ell = 0$ if $\ell \not\in \{\ell_1, \dots, \ell_k\}$;
then 
$$
\Psi_k (y_1, \dots, y_m, z_1, \dots, z_k) \coloneqq 
\left( h \left( z_i + \rho_{\ell_i}(z_i) y_{\ell_i} \right) \right)_{i=1,\dots,k}
$$
Taking the derivative with respect to $(y_{\ell_1}, \dots, y_{\ell_k})$, we 
obtain an onto linear map;
in particular the derivative $D\Psi_k(0, \dots, 0, z_1, \dots, z_k)$ 
is also onto, as we wanted to show.
\end{proof}

Fix some $\epsilon \in (0,\delta)$ with the property given by 
Fact~\ref{fa.transv}.
It follows from the parametric transversality theorem \cite[p.~79]{Hirsch} 
that for each nonempty $J \subset N$ there exists a residual subset of $R_J \subset (B(0,\epsilon))^m$
such that if $(y_1, \dots, y_m) \in R_J$ and $k = \# J$ then 
\begin{equation}\label{e.I_like_cake}
\Psi_k (y_1, \dots, y_m, \ \cdot \ ) \colon (S^{d-1})^{k} \to X^{k} \text{ is transverse to } G_J \, .
\end{equation}
Choose and fix a point $(y_{1}^*, \dots, y_{m}^*)$ in the residual set $\bigcap_{\emptyset \neq J\subset N} R_J \subset (B(0,\epsilon))^m$
close enough to $(0,\dots,0)$ so that the map $\tilde{h} \colon \bar{B}(0,1) \to X$
defined by
$$
\tilde{h} (z) \coloneqq  h \left( z + \sum_{\ell=1}^m \rho_\ell(z) y_{\ell}^* \right) 
$$
is a diffeomorphism.
We define an embedded disk $\tilde{K} \coloneqq \tilde{h}(\bar{B}(0,1))$.
Notice that the boundary $\partial \tilde{K}$ is contained in $V$,
and thus so is the symmetric difference $\tilde K \mathbin{\vartriangle} K$.

Let us check that $\partial \tilde{K}$ has the transverse $N$-hits property.
Fix any $x \in X$ and consider
$$
J = \big\{ j_1 < j_2 < \cdots < j_k \big\} \coloneqq \big\{ j \in N ; \; f^j(x) \in \partial \tilde{K} \big\} \, .
$$
We need to show that the hyperplanes
$$
H_i \coloneqq Df^{-j_i} \left(T_{f^{j_i}(x)} (\partial \tilde{K}) \right) , \quad (i=1,2,\dots,k)
$$
are independent.
Assume that $k > 1$, otherwise there is nothing to prove.
Let $\psi$ be the restriction of $\tilde{h}$ to the unit sphere,
so $\partial \tilde{K} = \psi (S^{d-1})$.
For each $i \in \{1,\dots,k\}$, let $z_{i}^* \coloneqq \psi^{-1} (f^{j_i}(x)) \in S^{d-1}$. 
Condition \eqref{e.I_like_cake} 
specialized to the point $(y_{1}^*,\dots,y_{m}^*)$ means that the map
$$
\psi_k \colon (S^{d-1})^{k} \to X^{k}
\quad\text{defined by }
\psi(z_1, \dots, z_k) \coloneqq (\psi(z_1), \dots, \psi(z_k)) 
$$	
is transverse to $G_J = f_J(\Delta_{k} X)$.
Equivalently, $f_J^{-1} \circ \psi_k \transverse \Delta_{k} X$.
In particular, since $f_J^{-1} \circ \psi_k (z_{1}^*,\dots, z_{k}^*) = (x,\dots,x)$,
the spaces
$$
\Im D (f_J^{-1} \circ \psi_k) (z_{1}^*, \dots, z_{k}^*) = H_1 \times \cdots \times H_k
\quad \text{and} \quad
T_{(x,\dots,x)}(\Delta_{k} X) = \Delta_{k} (T_{x} X)
$$
are transverse in $(T_{x} X)^{k}$.
By Lemma~\ref{l.transv_equivalence}, this means that
the hyperplanes $H_1$, \dots, $H_k$ are independent.
The proof of Lemma~\ref{l.transverse_hits} is concluded.
\end{proof}

\subsection{Fine stratifications}

Before proving Theorem~\ref{t.regularity},
we need to establish analogous regularity properties for 
certain auxiliary finer stratifications
that have a simpler local structure.

Suppose $K\subset X$ is a closed set with nonempty interior.
Define $m(K)$ as the least positive integer $m$ such that
$\bigcup_{j=0}^{m-1} f^{-j}(\interior K) = X$.
Notice that the numbers defined in \eqref{e.ell_plus} and \eqref{e.ell_minus}
satisfy the bounds:
$$
\ell^+(x) \le m(K)-1  \quad \text{and} \quad 
\ell^-(x) \le m(K)    \quad \text{for every } x \in X.
$$
Let  
$$
N(K) \coloneqq \{ j \in \Z \; ; -m(K) \le j \le m(K)-1 \} \, , 
$$
and for each $x \in X$, let
$$
\cM(x) \coloneqq \{j\in N(K) ; \; f^j(x) \in \partial K \} \, .
$$
Therefore the set $\cL(x)$ defined in \eqref{e.L}
equals $\cM(x) \cap (-\ell^-(x) , \ell^+(x))$. 

\medskip
Recall that the dynamical stratification associated to $K$ is the sequence
$$
X = X_0 \supset X_1 \supset \cdots \, , \quad \text{where }
X_i \coloneqq \{x \in X ;\; \#\cL(x) \ge i\} \, .
$$
The \emph{strata} of the stratification $(X_i)$ are defined as the connected components
of the nonempty sets $X_i \setminus X_{i+1}$, and form a partition of $X$.

We now define the \emph{fine dynamical stratification associated to $K$} as the sequence
$$
X = W_0 \supset W_1 \supset \cdots \, , \quad \text{where }
W_i \coloneqq \{x \in X ;\; \#\cM(x) \ge i\} \, .
$$
The corresponding strata are the connected components
of the nonempty sets $W_i \setminus W_{i+1}$.
Note that $W_i \supset X_i$ for every $i$. 
In other words, the strata of $(W_i)$ form a finer partition
than those of $(X_i)$. 
Also note that $W_{2m(K)+1} = \emptyset$.

The following lemma relates locally the two stratifications:

\begin{lemma}\label{l.frontier}
Letting $K_i \coloneqq K \cap X_i$ for each $i \ge 0$,
we have $\overline{W_i \setminus K_j} \cap K_j \subset W_{i+1}$ 
for any $i$, $j \ge 0$.
\end{lemma}

The proof of this lemma is somewhat similar to that of Lemma~\ref{l.loc_const}:

\begin{proof}
If $i<j$ then $K_j \subset W_{i+1}$ and the assertion becomes trivial.
So assume that $i \ge j$.
Consider a point $x \in \overline{W_i \setminus K_j} \cap K_j$;
let us show that $x \in W_{i+1}$.
Choose a sequence $(x_n)$ in $W_i \setminus K_j$ converging to $x$.
By passing to a subsequence, we can assume that $\cM(x_n)$, $\ell^+(x_n)$ and $\ell^-(x_n)$ 
are all independent of $n$.
By continuity of $f$ we have
\begin{equation}\label{e.3_cases}
\cM(x_n) \subset \cM(x)   \, , \quad
\ell^+(x_n) \le \ell^+(x) \quad \text{and} \quad
\ell^-(x_n) \le \ell^-(x) \, .
\end{equation}
Notice that the sets 
$$
\cL(x) = \cM(x) \cap (-\ell^-(x), \ell^+(x)) \quad \text{and} \quad
\cL(x_n) = \cM(x_n) \cap (-\ell^-(x_n), \ell^+(x_n))
$$
are different, because $\# \cL(x) \ge j > \# \cL(x_n)$.
Therefore at least one of the relations in \eqref{e.3_cases} is strict.
We consider the three possible cases:
\begin{itemize}
\item If $\cM(x_n) \subsetneqq \cM(x)$ then 
$\# \cM(x) > \# \cM(x_n) \ge i$, so $x \in W_{i+1}$.
\item If $\ell^+(x_n) < \ell^+(x)$ then $z: = f^{\ell^+(x_n)}(x) \not \in \interior K$;
on the other hand,
$$
z = \lim_{n \to \infty} f^{\ell^+(x_n)}(x_n) \in \overline{\interior(K)} \subset K \, ,
$$
so $z \in \partial K$.
This shows that $\ell^+(x_n) \in \cM(x) \setminus \cM(x_n)$,
and so by the previous case, $x \in W_{i+1}$.
\item If $\ell^+(x_n) < \ell^+(x)$ then $x \in W_{i+1}$ analogously.
\end{itemize}
This proves Lemma~\ref{l.frontier}.
\end{proof}

The following important lemma yields regularity properties for the stratification $(W_i)$:

\begin{lemma}\label{l.reg_strat_fine}
If $K$ is an embedded $d$-dimensional disk with the transverse $N(K)$-hits property
then the fine dynamical stratification $(W_i)$ associated to $K$ satisfies the following properties:
\begin{itemize}
\item $W_{d+1} = \emptyset$; 
\item for each $i\in \{0,1,\dots,d\}$, the pair $(W_i, W_{i+1})$ has the homotopy extension property.
\end{itemize}
\end{lemma}

Let us summarize what is involved in the proof of this lemma.
It follows from transversality that $W_{d+1} = \emptyset$ 
and moreover for each $i\in \{0,1,\dots,d\}$, 
every stratum $S \subset W_i \setminus W_{i+1}$
is a submanifold of codimension $i$ such that $\overline{S} \setminus S \subset W_{i+1}$.
We construct a vector field $\mathbf{v}$ on $X$
that is tangent to each strata and ``points to'' strata of higher codimension.
Despite being discontinuous, this vector field can be integrated
to a flow whose restriction to each strata is continuous
and moreover has the property that if a point in $W_i \setminus W_{i+1}$ is close to $W_{i+1}$
then its flow hits $W_{i+1}$ in small positive time.
Using this flow we construct the desired retractions in order to conclude that 
each pair $(W_i, W_{i+1})$ has the homotopy extension property.

\begin{proof}
Let $K$ be an embedded $d$-dimensional disk with the transverse $N$-hits property, where $N=N(K)$.
Let $(W_i)$ be the associated fine dynamical stratification.
Since no $d+1$ hyperplanes in a $d$-dimensional space can be independent,
we have $W_{d+1} = \emptyset$.

Fix a smooth map $\lambda_0 \colon X \to \R$ having $0$ as a regular value 
and such that the submanifold $\lambda_0^{-1}(0)$ is precisely $\partial K$.
Let $\lambda_j \coloneqq \lambda_0 \circ f^j$ for each $j \in N$.

Given $x \in X$, let 
$$
\sigma_1(x) \le \sigma_2(x) \le \cdots \le \sigma_{2m}(x)
$$
be the ordered list of the values $|\lambda_j(x)|$, where $j$ runs on the set $N$, with possible repetitions.
Notice that for each $i$, the function $\sigma_i$ is continuous and 
its zero set is $W_i$.

Fixed $x \in X$, let $i = i(x)$ be the least nonnegative integer such that $x \in W_i \setminus W_{i+1}$.
Let 
$$
\big \{j_1 = j_1(x) < j_2 = j_2(x) < \cdots < j_i = j_{i(x)}(x) \big\}
$$
be the set of times $j \in N$ such that $\lambda_j(x) = 0$, that is, $f^j(x) \in \partial K$.
By the transverse $N$-hits property, 
the functionals $D\lambda_{j_1}(x)$, \dots, $D\lambda_{j_i}(x)$ are linearly independent.
Therefore we can choose
smooth functions $\chi_1$, \dots, $\chi_{d-i}$ on a neighborhood $U(x)$ of $x$
such that 
$$
\psi_x \coloneqq \big( \lambda_{j_1} , \dots , \lambda_{j_i} , \chi_1, \dots, \chi_{d-i})
$$ 
is a diffeomorphism from $U(x)$ onto a subset of $\R^d$. 
(If $i=0$ then $\psi_x$ is an arbitrary chart around $x$.)
Let 
$$
\delta(x) \coloneqq \frac{\sigma_{i+1}(x)}{2} \, .
$$
Reducing the neighborhood $U(x)$ if necessary, 
we ensure that for every point $y \in U(x)$ we have
\begin{equation}\label{e.delta}
\max \big( |\lambda_{j_1}(y)|,  |\lambda_{j_2}(y)| , \dots ,  |\lambda_{j_i}(y)| \big) < \delta(x) < \sigma_{i+1}(y) \, ,
\end{equation}
where $\max \emptyset \coloneqq 0$.
In particular, $U(x) \subset W_i \setminus W_{i+1}$.

Next, take a finite subcover of the manifold $X$ by these neighborhoods:
$$
X = \bigcup_{\alpha=1}^\nu U_\alpha, \qquad U_\alpha \coloneqq U(x_\alpha).
$$

For each $i \in \{1,\dots, d\}$,
define the following (discontinuous) vector field on the euclidian space $\R^d$:
$$
\mathbf{u}_i(z_1, \dots, z_d) \coloneqq \big( -\sgn(z_1) ,\dots, -\sgn(z_i), 0, \dots, 0 \big),
$$
where $\sgn(t)$ is defined as $1$, $0$ or $-1$, depending on whether $t$ is positive, zero or negative, respectively.
Also let $\mathbf{u}_0 \coloneqq 0$.

For each $\alpha \in \{1,\dots, \nu\}$, let $i_\alpha \coloneqq i(x_\alpha)$
and let $\mathbf{v}_\alpha$ be the pull-back of the vector field $\mathbf{u}_{i_\alpha}$ 
under the diffeomorphism $\psi_\alpha \coloneqq \psi_{x_\alpha}$,
that is,
$$
\mathbf{v}_\alpha(y) \coloneqq [D\psi_\alpha(y)]^{-1} \big( \mathbf{u}_{i_\alpha}( \psi_\alpha(y) ) \big) \quad \text{for all } y \in U_\alpha \, .
$$
See Fig.~\ref{f.fields}.
Notice that for each stratum $S$ intersecting $U_\alpha$,
the vector field $\mathbf{v}_\alpha$ restricted to the submanifold $S\cap U_\alpha$
is smooth and tangent to it.

\newcommand{\curva}{.06}
\newcommand{\curvb}{-.2}
\newcommand{\flecha}{.25}
\begin{figure}[htp]
	\begin{tikzpicture}[scale=1.5,font=\scriptsize]
		\draw[blue,thick,variable=\u,domain=-1.25:3.5] plot({\u},{\curva*\u*\u});
		\draw[blue,thick,variable=\v,domain=-1:1] plot({\curvb*\v*\v},{\v});
		
		\foreach \u in {-.75,.75}{
			\draw[gray,variable=\v,domain=-.75:.75] plot({\u+\curvb*\v*\v},{\v+\curva*\u*\u});
		}
		\foreach \v in {-.75,.75}{
			\draw[gray,variable=\u,domain=-.75:.75] plot({\u+\curvb*\v*\v},{\v+\curva*\u*\u});
		}
		\foreach \u in {1.5,3}{
			\draw[gray,variable=\v,domain=-.75:.75] plot({\u+\curvb*\v*\v},{\v+\curva*\u*\u});
		}	
		\foreach \v in {-.75,.75}{
			\draw[gray,variable=\u,domain=1.5:3] plot({\u+\curvb*\v*\v},{\v+\curva*\u*\u});
		}

		\foreach \u in {-.5,0,.5}{
			\foreach \v in {-.5,0,.5}{
				\draw[fill] ({\u+\curvb*\v*\v},{\v+\curva*\u*\u}) circle(.03);
				\ifthenelse{\equal{\u}{0} \and \equal{\v}{0}}{}{
					\draw[very thick, ->] ({\u+\curvb*\v*\v},{\v+\curva*\u*\u})--++(-2*\u*\flecha-4*\curvb*\v*\v*\flecha,-2*\v*\flecha-4*\curva*\u*\u*\flecha);}
			}
		}

		\foreach \u in {1.75,2.25,2.75}{
		 	\foreach \v in {-.5,0,.5}{
		 		\draw[fill] ({\u+\curvb*\v*\v},{\v+\curva*\u*\u}) circle(.03);
				\ifthenelse{\equal{\v}{0}}{}{
					\draw[very thick, ->] ({\u+\curvb*\v*\v},{\v+\curva*\u*\u})--++(-4*\curvb*\v*\v*\flecha,-2*\v*\flecha);
				}
			}
		}
		
		\node[below right] at (0,0) {$x_\alpha$};
		\node[below right] at (2.25,\curva*2.25*2.25) {$x_\beta$};		
		\node[left] at (-.75+\curvb*.75*.75, .75+\curva*.75*.75) {$U_\alpha$};
		\node[right] at (3+\curvb*.75*.75, .75+\curva*3*3) {$U_\beta$};

	\end{tikzpicture}
	\caption{The vector fields $\mathbf{v}_\alpha$ and $\mathbf{v}_\beta$ around points $x_\alpha \in W_2$ and $x_\beta \in W_1 \setminus W_2$.}\label{f.fields}
\end{figure}
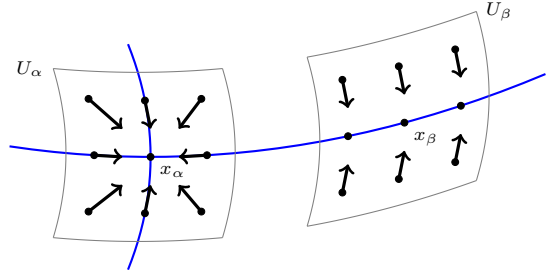

Let $(\rho_\alpha)$ be a smooth partition of unity subordinate the open cover $(U_\alpha)$.
Define a (discontinuous) vector field $\mathbf{v}$ on $X$ by the formula:
$$
\mathbf{v} (y) \coloneqq \sum_{\alpha ; \;  U_\alpha \ni y} \rho_\alpha(y) \mathbf{v}_\alpha(y) \, .
$$
Then the restriction of $\mathbf{v}$ to each stratum is smooth 
(and in particular, locally Lipschitz) 
and tangent to it.

Let $\delta_\alpha \coloneqq \delta(x_\alpha)$ and $\delta_* \coloneqq \min_\alpha \delta_\alpha$.
We claim that the vector field $\mathbf{v}$ has following property:
\begin{equation}\label{e.it_s_a_trap}
D\lambda_j(y)(\mathbf{v} (y)) = -\sgn(\lambda_j(y))
\text{ for every } j \in N \text{ such that } |\lambda_j(y)| < \delta_* \, .
\end{equation}
Indeed, assume that $j \in N$ and $|\lambda_j(y)| \le \delta_*$.
Fix any $\alpha$ such that $U_\alpha \ni y$.
Let $i = i_\alpha$ and $j_1 = j_1(x_\alpha) < \cdots < j_i = j_{i_\alpha}(x_\alpha)$.
Then, by \eqref{e.delta},
$$
\{ k \in N ;\; |\lambda_k(y)| < \delta_\alpha \} = \{ j_1, \dots, j_i \} \, .
$$
The index $j$ belongs to this set, because $\delta_* \le \delta_\alpha$.
It now follows from the definition of $\mathbf{v}_\alpha$ that
$$
D\lambda_j(y)(\mathbf{v}_\alpha (y)) = -\sgn(\lambda_j(y)) \, .
$$
Since this equality holds for every $\alpha$ such that $U_\alpha \ni y$,
we conclude that \eqref{e.it_s_a_trap} holds.

\medskip

From now on, let $i \in \{0,\dots,d\}$ be fixed.
We need to find a retraction 
\begin{equation}\label{e.g_i}
g_i \colon W_i \times I \to W_i \times \{0\} \widecup W_{i+1} \times I.
\end{equation}
If $i=d$ then $g_d(x,s)$ is simply $(x,0)$.
Let us assume that $i<d$.

For each $x \in W_i \setminus W_{i+1}$,  
the ODE $x'(t) = \mathbf{v}( x(t) )$ with initial value $x(0)=x$
has a unique solution $x(t) = \phi_i(x, t)$ taking values in $W_i \setminus W_{i+1}$.
Let $\tau_i(x) \in (0,\infty]$ be the supremum of the maximal interval where this solution is defined.
Note that:
\begin{equation}\label{e.hitting_time}
x \in W_i \setminus W_{i+1}, \ 
\sigma_{i+1}(x) < \delta_* \ \Rightarrow \ \tau(x) = \sigma_{i+1}(x);
\end{equation}
indeed once the quantity $\sigma_{i+1}$ is less than $\delta_*$, 
by \eqref{e.it_s_a_trap} it decreases with unit speed until $W_{i+1}$ is hit.

Property \eqref{e.hitting_time} has the following consequences:
\begin{itemize}
\item The map $\tau_i$ is continuous, and 
can be continuously extended to a map $\bar{\tau}_i$ on $W_i$
which vanishes on $W_{i+1}$.

\item If $x\in W_i \setminus W_{i+1}$ is such that 
$\tau_i(x) < \infty$ then 
$$
\pi_i(x) \coloneqq \phi_i(x,  \tau_i(x)-) = \lim_{t \to \tau_i(x)-} \phi_i(x, t)
$$
exists and belongs to $W_{i+1}$.
Moreover, the map $\pi_i$ is continuous.
\end{itemize}
It follows that the map 
$\bar{\phi}_i \colon W_i \times [0, \infty) \to W_i$
defined by
$$
\bar{\phi}_i (x,t) \coloneqq
\begin{cases}
	x			&\text{if $x \in W_{i+1}$ or $t=0$,} \\
	\phi_i(x,t)	&\text{if $x \not\in W_{i+1}$ and $0 < t < \tau_i(x)$,} \\
	\pi_i(x)	&\text{if $x \not\in W_{i+1}$ and $t \ge \tau_i(x)$}
\end{cases}
$$
is continuous.
(This map can be viewed as the positive-time flow
generated by the vector field $\mathbf{v}$ on $W_i$
and with $W_{i+1}$ as an absorbing barrier.)

Define a continuous map $\bar{g}_i \colon W_i \times I \times I \to W_i \times I$ by
\begin{equation}\label{e.deformation}
\bar{g}_i (x,s,t) \coloneqq \left( \bar{\phi}_i \big(x, \min\{\bar{\tau}_i(x),s,t\} \big), s - \min\{\bar{\tau}_i(x),s,t\} \right) \, .
\end{equation}
(This map can be viewed as the flow
generated by the vector field $(\mathbf{v},-1)$ on $W_i \times I$
and with $W_i \times \{0\} \cup W_{i+1}\times I$ as an absorbing barrier.)
Then the map $g_i \coloneqq \bar{g}_i (\mathord{\cdot},\mathord{\cdot},1)$ is 
the desired retraction \eqref{e.g_i}.
This proves that $(W_i, W_{i+1})$ has the homotopy extension property.
\end{proof}

\begin{scholium}\label{scholium}
The retractions $g_i \colon W_i \times I \to W_i \times \{0\} \widecup W_{i+1} \times I$
constructed in the proof of Lemma~\ref{l.reg_strat_fine}
have the following property:
if $g_i(x,s) = (x',s')$ then $\cL(x') \supset \cL(x)$.
\end{scholium}

\begin{proof}
We will make use of the map $\bar{g}_i$ defined by \eqref{e.deformation},
which is actually a deformation retraction.
Let $(x,s) \in W_i \times I$ and let $(x',s') = g_i(x,s)$.
Assume that $i<d$ and $(x,s) \not\in W_i \times \{0\} \widecup W_{i+1} \times I$, otherwise there is nothing to prove.
For $t\in I$, let $\xi_i(x,s,t)$ be first coordinate of $\bar{g}_i(x,s,t)$.
Then $\xi_i(x,s,t) \in W_i \setminus W_{i+1}$ for every $t \in [0, t_*)$, where $t_*\coloneqq\min\{\bar{\tau}_i(x),s\} > 0$.
By connectedness we conclude that $\cM(\xi_i(x,s,t))$ is independent of $t \in [0, t_*)$.
It follows that
$$
\cM (\xi_i(x,s,t_*) ) \supset \cM ( \xi_i(x,s,0) ) \, , \quad \text{that is,}\quad
\cM(x') \supset \cM(x) \, .
$$
A similar argument shows that $\ell^+(x') \ge \ell^+(x)$ and  $\ell^-(x') \ge \ell^-(x)$.
Therefore $\cL(x') \supset \cL(x)$, as announced. 
\end{proof}

\subsection{End of the proof}

\begin{lemma}\label{l.reg_strat_2}
If $K$ is an embedded $d$-dimensional disk with the transverse $N(K)$-hits property
then it is regular. 
\end{lemma}

\begin{proof}
Apply Lemma~\ref{l.reg_strat_fine} to the disk $K$ and the associated
fine dynamical stratification $(W_i)$.
Then $X_{d+1} \subset W_{d+1} = \emptyset$,
which is the first regularity property that we need to check. 

For each $i \in \{0, \dots, d\}$, Lemma~\ref{l.reg_strat_fine} 
gives us a retraction
$$
g_i \colon W_i \times I \to W_i \times \{0\} \widecup W_{i+1} \times I \, .
$$
Denote $K_i \coloneqq K \cap X_i$.

Let $j \in \{0,1,\dots, d\}$ be fixed.
We will explain how to find a retraction 
\begin{equation}\label{e.wanted}
r \colon K_j \times I \to K_j \times \{0\} \widecup K_{j+1} \times I \, .
\end{equation}
First notice that, for any $i$,
\begin{equation}\label{e.fit}
g_i \left( (K_j \cap W_i) \times I \right) \subset K_j \times \{0\} \widecup (K_j \cap W_i) \times I \, .
\end{equation}
Indeed this follows from the property of $g_i$
provided by the Scholium~\ref{scholium}.
Define a nested sequence of closed sets 
$$
K_j \times I = Q_j \supset Q_{j+1} \supset \cdots \supset Q_d \supset Q_{d+1} = K_j \times \{0\} \widecup K_{j+1} \times I \, .
$$
by 
$$
Q_i \coloneqq K_j \times \{0\} \widecup \left( K_{j+1} \cup (K_j \cap W_i) \right) \times I \, .
$$
For each $i \ge j$, the set
$$
Q_i \setminus Q_{d+1} = (K_j \cap W_i \setminus K_{j+1}) \times (I \setminus \{0\});
$$
is contained in $(K_j \cap W_i) \times I$;
therefore by \eqref{e.fit} we can define a map 
$$
h_i \colon Q_i \to Q_{i+1} \quad \text{by} \quad
h_i = 
\begin{cases}
	\id &\text{on } Q_{d+1} \, , \\
	g_i &\text{on } Q_i \setminus Q_{d+1} \, .
\end{cases}
$$
Notice that
\begin{align*}
\overline{Q_i \setminus Q_{d+1}} \widecap Q_{d+1} 
&\enspace{\subset}\enspace \left( \, \overline{W_i \setminus K_{j+1}} \times I \right) \widecap Q_{d+1} \\
&\enspace{\subset}\enspace W_i \times \{0\} \widecup \left( \, \overline{W_i \setminus K_{j+1}} \cap K_{j+1} \right) \times I  \\
&\enspace{\subset}\enspace W_i \times \{0\} \widecup W_{i+1} \times I \quad \text{(by Lemma~\ref{l.frontier}),}
\end{align*}
and so $g_i = \id$ on these sets.
It follows that $h_i$ is continuous.

The desired retraction \eqref{e.wanted} is 
$$
r \coloneqq  h_{d+1} \circ h_d \circ \cdots \circ h_{i+1} \circ h_i \colon Q_i \to Q_{d+1}  \, .
$$
This shows the regularity of the dynamical stratification $(X_i)$.
\end{proof}

\begin{proof}[Proof of Theorem~\ref{t.regularity}]
Given any point $x \in X$ and any open neighborhood $V \ni x$,
let $K_1$, $K_2 \subset V$ be embedded $d$-dimensional disks containing $x$ in its interior,
with $K_1 \subset \interior K_2$.
Let $N = N(K_1)$.
By Lemma~\ref{l.transverse_hits}, there exists a 
disk $K$ with the transverse $N$-hits property such that 
$K \mathbin{\vartriangle} K_2 \subset V \setminus K_1$.
Note that $x \in K \subset V$
and so the sets $K$ obtained in this way form a basis of neighborhoods of $x$.
Moreover, $K \supset K_1$
and in particular $N(K) \subset N(K_1)$.
So $K$ also has the transverse $N(K)$-hits property,
and therefore it is a regular set by Lemma~\ref{l.reg_strat_2}.
\end{proof}


\end{document}